\documentclass{article}

\usepackage{arxiv}
\usepackage{amsmath,amssymb,amsthm,mathtools,subcaption}
\usepackage[utf8]{inputenc} % allow utf-8 input
\usepackage[T1]{fontenc}    % use 8-bit T1 fonts
\usepackage{hyperref}       % hyperlinks
\usepackage{url}            % simple URL typesetting
\usepackage{booktabs}       % professional-quality tables
\usepackage{amsfonts}       % blackboard math symbols
\usepackage{nicefrac}       % compact symbols for 1/2, etc.
\usepackage{microtype}      % microtypography
\usepackage{lipsum}
\usepackage{graphicx}
\usepackage{xcolor}
\usepackage{mathrsfs}
\usepackage[noabbrev]{cleveref}
\usepackage[titletoc, title]{appendix}
\usepackage{changes}
\colorlet{Changes@Color}{blue}

\graphicspath{ {./images/} }

\Crefname{appsec}{appendix}{appendices}
\newcommand{\R}{\mathbb{R}}
\newcommand{\C}{\mathbb{C}}
\title{Rigorous justification of the Whitham modulation theory for equations of NLS type}
\newtheorem{Theorem}{Theorem}
\newtheorem{Lemma}{Lemma}
\newtheorem{proposition}{Proposition}
\theoremstyle{plain}
\newtheorem{corollary}{Corollary}
\newtheorem{definition}{Definition}
\theoremstyle{remark}
\newtheorem{remark}{Remark}
\author{
	W.A. Clarke \\
	School of Mathematics and Statistics\\
	University of Sydney\\
	\texttt{wcla7359@uni.sydney.edu.au} \\
	\And
	R. Marangell \\
	School of Mathematics and Statistics\\
	University of Sydney\\
	\texttt{robert.marangell@sydney.edu.au} \\
}

\begin{document}
	\maketitle
	\begin{abstract}
		We study the modulational stability of periodic travelling wave solutions to equations of nonlinear Schr{\"o}dinger type. In particular, we prove that the characteristics of the quasi-linear system of equations resulting from a slow modulation approximation satisfy the same equation, up to a change in variables, as the normal form of the linearized spectrum crossing the origin. This normal form is taken from  \cite{LBJM2019}, where Leisman et al. compute the spectrum of the linearized operator near the origin via an analysis of Jordan chains. We derive the modulation equations using Whitham's formal modulation theory, in particular the variational principle applied to an averaged Lagrangian. We use the genericity conditions assumed in the rigorous theory of \cite{LBJM2019} to direct the homogenization of the modulation equations. As a result of the agreement between the equation for the characteristics and the normal form from the linear theory, we show that the hyperbolicity of the Whitham system is a necessary condition for modulational stability of the underlying wave.
	\end{abstract}

	% keywords can be removed

	\section{Introduction}
	In this paper, we consider the modulational stability of periodic solutions to equations of nonlinear Schr{\"o}dinger type
	\begin{align}
	i\psi_{t} &= \psi_{xx} + \zeta f(\lvert\psi\rvert^{2})\psi \label{eq:1DNLS}
	\end{align}
	under perturbations in $ L_{2}(\R) $. The nonlinearity $ f(\lvert\psi\rvert^{2}) $ is arbitrary, but assumed to be well-behaved (double-integrable). This equation has a rich history of study; the cubic nonlinear Schr{\"o}dinger equation, $ f(\lvert\psi\rvert^{2}) = \pm\lvert\psi\rvert^{2}$, describes the envelope of a slowly modulated carrier wave in a dispersive medium \cite{SS1999,AS1981}. This equation, and others of NLS-type with higher order nonlinearities, arise in the study of a plethora of physical systems, including: water waves \cite{Zakharov1968,HO1972}; nonlinear optics \cite{Agrawal2013,HM2003}; plasma physics \cite{Chen2016,LTE2019,MOMT1976}; and Bose-Einstein condensates \cite{Gross1961,Pitaevskii1961}. \par
	Since the dynamics of \cref{eq:1DNLS} exhibit linear, nonlinear and modulatory behaviour, the literature includes analyses of linearized, orbital and modulational stability. If one chooses a suitable potential $ f(\lvert \psi\rvert^{2}) $ such that \cref{eq:1DNLS} is integrable, then it is possible to give an explicit description of the spectrum. The cubic nonlinear Schr{\"o}dinger equation is the best example of this \cite{BDN2011,DS2017}, however relying on integrability is not necessary \cite{GLCT2017}. Rowlands determined the spectral stability of stationary periodic solutions of the cubic nonlinear Schr{\"o}dinger equation subject to long-wavelength disturbances \cite{Rowlands1974}, and in so doing demonstrated modulational instability in the focusing case. Alfimov, Its and Kulagin \cite{AIK1990} constructed the homoclinic orbit for an unstable, spatially periodic solution to the focusing nonlinear Schr{\"o}dinger equation, essentially providing a nonlinear description of the modulational stability of this type of solution. Using the general methods of \cite{GSS1987,GSS1990}, Gallay and H\v{a}r\v{a}gu{\c{s}} proved that quasiperiodic, small-amplitude solutions to the cubic nonlinear Schr{\"o}dinger equation are orbitally stable within the class of solutions having the same period and Floquet multipler \cite{GH2007A}. They further proved that these solutions are linearly stable under bounded perturbations in the defocusing case, but linearly unstable in the focusing case. In \cite{GH2007B}, Gallay and H\v{a}r\v{a}gu{\c{s}} extend the orbital stability results in \cite{GH2007A} to solutions of any amplitude. \par
	This paper deals with modulational stability of \cref{eq:1DNLS}, that is, the spectral stability subject to long-wavelength perturbations. Rigorously speaking, this amounts to expanding the spectrum of the linearized operator in a neighbourhood of the origin in the spectral plane. Whitham modulation theory \cite{Whi1965A,Whi1965B,Whi1967,Whi1970,Whi1999} provides a formal procedure in which the modulational stability is computed by considering the hyperbolicity of a system of PDEs called the Whitham modulation equations. These modulation equations arise from an asymptotic expansion of the governing PDE via $ (x,t) \mapsto (\epsilon x,\epsilon t) $ along with a WKB approximation of the solution. To $ O(\epsilon) $, these equations are an homogeneous system of PDEs in terms of the slowly-varying parameters of the original PDE. Proving that the Whitham modulation theory accurately predicts the results of the rigorous analysis of the linearized spectrum is non-trivial, and to our knowledge is an open problem in the general case. Our present analysis is influenced by the examples of where this has been done, such as: the nonlinear Klein-Gordon equation \cite{Jones14}; the generalized Korteweg-de Vries equation \cite{JZ2010}; systems of viscous conservation laws \cite{Serre2005,OZ2006}; and a viscous fluid conduit equation \cite{JP2020}.\par
	In addition to the WKB approximation and asymptotic expansion, Whitham proves that there are several other, equivalent methods for deriving the Whitham modulation equations \cite{Whi1970,Whi1999}. In particular, these are: the variational principle applied to an averaged Lagrangian; and averaging and two-timing a system of conservation laws. As a result of unwieldy algebra that arises in the modulation equations, one often chooses the method which provides the simplest derivation of the Whitham modulation equations, thereafter making a change of variables in order to prove that the hyperbolicity of the Whitham system is equivalent to the linearized stability theory. When the modulational instability criterion from the linearized stability theory is given in terms of the derivatives of special quantities of the original PDE such as the period, mass and momentum, a useful technique is to introduce a classical action variable $ W $, since the derivatives of $ W $ are related to these special quantities. This is the case for both the nonlinear Klein-Gordon equation \cite{Jones14} and the generalized Korteweg-de Vries equation \cite{BJ2010,JZ2010}. This approach relies on the symmetry properties of Lagrangian systems, and Whitham in fact derives the modulation equations using the averaged Lagrangian method in the cases of both of these PDEs (only the non-generalized case for KdV) \cite{Whi1999}. \Cref{eq:1DNLS} admits a Lagrangian (cf. \cref{eq:nlsLagrangian}), and so we follow the averaged Lagrangian method in order to more easily derive the Whitham modulation equations. If, however, there is not an obvious Lagrangian but the travelling wave ODE is integrable, one can compute the tangent space to the manifold of travelling wave solutions and construct the kernel of the linearized operator from a basis of the tangent space. In \cite{Serre2005}, Serre performs this exact procedure to prove that the hyperbolicity criterion of the Whitham modulation equations agrees with the rigorous modulational stability analysis (as determined from an Evans function expansion) of a system of scalar, viscous conservation laws \cite{OZ2003}, which is later extended to the multi-dimensional case \cite{OZ2006}. For the conduit equation, which is neither completely integrable nor Langrangian in structure, Johnson and Perkins rigorously derive a matrix encoding the leading order asymptotics of the spectral curves at the origin via Floquet-Bloch theory and spectral perturbation theory \cite{JP2020}. The entries of this matrix agree exactly with the coefficients of the Whitham system, allowing for a more direct comparison of modulational instability criteria.\par
	There is a well-developed Whitham modulation theory for the nonlinear Schr{\"o}dinger equation; D{\"u}ll and Schneider proved that solutions to the Whitham modulation equations that are analytic in a strip are a valid approximation of spatial and temporal modulations of periodic wave solutions of the cubic nonlinear Schr{\"o}dinger equation \cite{DS2009}. A similar result exists for solutions to multiphase Whitham modulation equations in the context of coupled NLS equations \cite{BKS2020}. In the context of Sobolev spaces, proving that solutions to the Whitham modulation equations are a valid approximation to modulations of a periodic wavetrain requires spectral stability of the underlying wavetrain, and so Bridges et al. restricted their analysis to hyperbolic Whitham modulation equations for the defoccusing NLS equation \cite{BKZ2021}. In \cite{Bridges2015,BR2019}, the authors investigate the transition between a hyperbolic and elliptic system of Whitham modulation equations for cubic NLS in the single phase case and a coupled NLS system in the multiphase case. In \cite{Kamchatnov2000}, Kamchatnov provides techniques for calculating the Riemann invariants of the cubic nonlinear Schr{\"odinger} equation, from which the hyperbolicity of the Whitham modulation equations can be determined. Moreover, the Riemann invariants of a Whitham modulation system encode the asymptotic description of a dispersive shock wave; El and Hoefer provide an extensive discussion on this connection in \cite{EH2016}, as well as an analysis of dispersive shock waves in the case of cubic NLS and numerical simulations of dispersive shock wave behaviour for \cref{eq:1DNLS}. El and Hoefer also note that applying the Madelung transformation ($ \psi = \sqrt{\rho}e^{i\phi} $) to \cref{eq:1DNLS} yields the hydrodynamic system:
	\begin{align}
	\rho_{t} + (\rho u)_{x} &= 0 \label{eq:hd1}\\
	(\rho u)_{t} + \left(\rho u^{2} + P(\rho)\right)_{x} &= \left(\rho\left(\ln\rho\right)_{xx}\right)_{x},\label{eq:hd2}
	\end{align}
	with $ P(\rho) = \int_{0}^{\rho}2\zeta sf'(s)ds $ and $ u = -\phi_{x} $. Serre's results \cite{Serre2005} are not proven to extend to higher order dispersive terms such as $ \rho\left(\ln\rho\right)_{xx} $. We also point out that, as an intermediate step in deriving \cref{eq:hd1,eq:hd2}, we can also write \cref{eq:1DNLS} as an Euler-Korteweg system in Eulerian coordinates \cite{BGMR2016}:
	\begin{align}
		\rho_{t} + (\rho u)_{x} &= 0 \label{eq:EK1}\\
		u_{t} + uu_{x} + \left(\delta \mathscr{E}\right)_{x} &= 0, \label{eq:EK2}
	\end{align}
	with $ \mathscr{E} = -2\zeta F(\rho) + \frac{\rho_{x}^{2}}{2\rho} $ and $ \delta \mathscr{E} = \mathscr{E}_{\rho} - \partial_{x}\mathscr{E}_{\rho_{x}}$. Benzoni-Gavage et al. proved that weak hyperbolicity of the Whitham modulation equations associated with the Euler-Korteweg system is a necessary condition for spectral stability \cite[Theorem 1]{BGNR2014}. We distinguish our method from \cite{BGNR2014,Serre2005} by not recasting \cref{eq:1DNLS} as either the hydrodynamic system in \cref{eq:hd1,eq:hd2} nor as the Euler-Korteweg system in \cref{eq:EK1,eq:EK2}, but instead we directly compute the polynomial whose roots are the characteristics of the linearized Whitham modulation equations.\par
	The main result of this paper, \Cref{th:maintheorem}, shows that the Whitham modulation theory for \cref{eq:1DNLS} correctly predicts the rigorous spectral stability results contained in \cite{LBJM2019}. Leisman et. al. relate spectral instability of periodic travelling wave solutions of \cref{eq:1DNLS} near the origin of the spectral plane to the breakup of the generalized kernel of the operator of the linearized problem. The authors express the genericity conditions on the Jordan chains of the linearized operator in terms of matrices whose entries are moments of the travelling wave solutions. They use these matrices to derive a normal form for the spectrum of the linearized operator at the origin subject to both longitudinal and transverse perturbations. We give a brief summary of their results, before applying Whitham's averaged Lagrangian method to derive the Whitham modulation equations. We use the genericity conditions given in \cite{LBJM2019} to homogenize the modulation equations, turning them into a quasi-linear system of four equations in four slowly-modulated parameters of the travelling wave solutions. The characteristics of this system are the zeros of a quartic equation which, upon changing variables, we identify as the normal form for the four continuous bands of spectrum emerging from the origin.\par 
	We have organized this paper into two main sections. In \Cref{sec:s2}, we introduce the relevant spectral stability results of \cite{LBJM2019}, leading to a rigorous criterion for modulational instability in \Cref{cor:c1}. We begin \Cref{sec:s31} with a general overview of Whitham's averaged Lagrangian approach to deriving the modulation equations. We then apply this theory to \cref{eq:1DNLS} in \Cref{sec:s32}, from which we conclude in \Cref{cor:c2} that the Whitham theory criterion for modulational instability agrees exactly with the rigorous spectral analysis of \Cref{sec:s2}. We conclude the paper with a discussion in \Cref{sec:s4} as well as, to the best of our knowledge, several open problems in the area. We include some necessary, though long calculations, as well as, we hope, some useful identities in \Cref{appendix:app1,appendix:app2,appendix:app3,appendix:app4,appendix:app5,appendix:app6}. %{\color{blue}RM: Do we want to include a specific description of what is in each appendix?}
	\section{Modulational stability from the linear theory}\label{sec:s2}
	In this section, we follow the analysis of the modulational stability problem contained in \cite{LBJM2019}.\par
	We seek travelling wave solutions of \cref{eq:1DNLS} in the form
	\begin{align*}
	\psi(x,t) &= e^{i\omega t}\phi(x+ct).
	\end{align*}
	Writing $ y = x+ct $, we express $ \phi $ in polar coordinates
	\begin{align}
	\phi(y) &= \exp\left(i\left(\theta_{0} + \frac{cy}{2} + S(y+y_{0})\right)\right)A(y+y_{0}). \label{eq:NLSphi}
	\end{align}
	Substituting \cref{eq:NLSphi} into \cref{eq:1DNLS} and equating real and imaginary parts we have:
	\begin{align}
	&2A_{y}S_{y} + AS_{yy} = 0 \label{eq:NLSphaseode}\\
	&A_{yy} = -\left(\omega + \frac{c^{2}}{4}\right)A + AS_{y}^{2} - \zeta f(A^{2})A. \label{eq:NLSmodode}
	\end{align}
	Integrating \cref{eq:NLSphaseode} yields
	\begin{align}
	S_{y} &= \frac{\kappa}{A^{2}}, \label{eq:NLSSode}
	\end{align}
	which can be substituted into \cref{eq:NLSmodode} and upon integrating we have:
	\begin{align}
	A_{y}^{2} &= 2E - \left(\omega + \frac{c^{2}}{4}\right)A^{2} - \frac{\kappa^{2}}{A^{2}} - \zeta F(A^{2}). \label{eq:NLSode}
	\end{align}
	The symmetries of \cref{eq:1DNLS} allow us to eliminate several of the seven parameters $ E, \omega, c, \kappa, \zeta, y_{0},\theta_{0} $. In particular, \cref{eq:1DNLS} is invariant under the transformations (assume $ \alpha \in \R $):
	\begin{itemize}
		\item $ \psi(x,t) \mapsto \psi(x,t)e^{i \alpha} \quad$ (phase invariance);\\
		\item $ \psi(x,t) \mapsto \psi(x + \alpha,t) \quad$ (translation invariance);\\
		\item $ \psi(x,t) \mapsto \psi(x + \alpha t,t)e^{-i\big(\frac{\alpha}{2}x + \frac{\alpha^{2}}{4}t\big)} \quad$ (Galilean invariance).
	\end{itemize}
	Phase and translation invariance means that we can eliminate $ \theta_{0} $ and $ y_{0} $ from \cref{eq:NLSphi}. Furthermore, by Galilean invariance we can reduce $ \psi $ to:
	\begin{align}
	\psi(x,t) &= e^{i\left(\omega + \frac{c^{2}}{4}\right)t}e^{iS(x)}A(x). \label{eq:reducedpsi}
	\end{align}
	Both \cref{eq:reducedpsi} and \cref{eq:NLSode} suggest that we can absorb $ \frac{c^{2}}{4} $ into $ \omega $ (or equivalently set $ c = 0 $), eliminating $ c $ from the equations. This leads us to make the following definition about the parameter space $ \Omega $:
	\begin{definition}[Definition 1 \cite{LBJM2019}]\label{def:d1} We define the parameter domain $ \Omega $ as the open set of parameter values $ (E,\kappa,\omega,\zeta) $ such that:
		\begin{itemize}
			\item $ \kappa > 0; $\\
			\item The function $ P(A) = 2E - \omega A^{2} - \frac{\kappa^{2}}{A^{2}} - \zeta F(A^{2}) $ has two positive, real, simple roots $ a_{-},\; a_{+} $ with $ a_{-} < a_{+} $ and $ P(A) $ is real and positive for $ A \in (a_{-},a_{+}) $. 
		\end{itemize}
	\end{definition}
	As explained in \cite[Remark 1]{LBJM2019}, parameter values in $ \Omega $ represent the most generic case and allow us to freely differentiate with respect to the parameters. The cases where $ \kappa = 0 $ or $ R(A) $ has higher order roots correspond to degenerate cases which require their own analyses. We assume henceforth that our parameters are taking values in $ \Omega $.\par
	We are interested in periodic solutions $ A $ to \cref{eq:NLSSode,eq:NLSode}. We make the assumption $ A(0) = a_{-},\; A(\frac{T_{*}}{2}) = a_{+} $, which we can always guarantee by translating the wave. Under these conditions, the function $ \phi $ has the quasi-periodic boundary conditions:
	\begin{align}
	\phi(T_{*}) &= \phi(0)e^{i\eta} \label{eq:qpL}\\
	\phi_{y}(T_{*}) &= \phi_{y}(0)e^{i\eta},
	\end{align}
	where the period $ T_{*} $ and the quasi-momentum $ \eta = S(T_{*}) - S(0) $ are functions of the parameters $ E,\kappa,\omega,\zeta $. These quantities, and the mass $ M $, can be expressed as derivatives of the classical action:
	\begin{align}
	K(E,\kappa,\omega,\zeta) &= \int_{0}^{T_{*}}A_{y}^{2}dy = 2\int_{a_{-}}^{a_{+}}\sqrt{2E - \omega A^{2} - \frac{\kappa^{2}}{A^{2}} - \zeta F(A^{2})}dA. \label{eq:classicalAction}
	\end{align}
	In particular, we have:
	\begin{align}
	\begin{split}
	T_{*}(E,\kappa,\omega,\zeta) &= 2\int_{a_{-}}^{a_{+}}\frac{1}{A_{y}}dA = 2\int_{a_{-}}^{a_{+}}\frac{1}{\sqrt{2E - \omega A^{2} - \frac{\kappa^{2}}{A^{2}} - \zeta F(A^{2})}}dA\\
	&= \frac{\partial K}{\partial E} \label{eq:TNLS}
	\end{split}\\
	\begin{split}
	\eta(E,\kappa,\omega,\zeta) &= S(T_{*}) - S(0) = 2\int_{a_{-}}^{a_{+}}\frac{\kappa}{A^{2}A_{y}}dA\\
	&= 2\int_{a_{-}}^{a_{+}}\frac{\kappa}{A^{2}\sqrt{2E - \omega A^{2} - \frac{\kappa^{2}}{A^{2}} - \zeta F(A^{2})}}dA\\
	&= -\frac{\partial K}{\partial\kappa} \label{eq:etaNLS}
	\end{split}\\
	\begin{split}
	M(E,\kappa,\omega,\zeta) &= 2\int_{a_{-}}^{a_{+}}\frac{A^{2}}{A_{y}}dA  = 2\int_{a_{-}}^{a_{+}}\frac{A^{2}}{\sqrt{2E - \omega A^{2} - \frac{\kappa^{2}}{A^{2}} - \zeta F(A^{2})}}dA\\
	&= -2\frac{\partial K}{\partial\omega}. \label{eq:MNLS}
	\end{split}
	\end{align}
	We now consider the linearized problem under the perturbation:
	\begin{align*}
	\psi(y,t) &= e^{i\omega t}\left(\phi(y) + \epsilon e^{i[S(y+y_{0}) + \frac{cy}{2} + \theta_{0}]}W(y,t)\right).
	\end{align*}
	Since the phase $ \exp(i[S(y+y_{0}) + \frac{cy}{2} + \theta_{0}])$ is also present in $ \phi(y) $ (\cref{eq:NLSphi}), then we can consider $ W $ as a complex-valued perturbation of $ A $. At $ O(\epsilon) $ we have:
	\begin{align}
	iW_{t} = W_{yy} + \omega W - S_{y}^{2}W + \zeta f(A^{2})W + 2\zeta f'(A^{2})A^{2}\mathrm{Re}(W) + i(S_{yy}W + 2S_{y}W_{y}). \label{eq:linearizedNLS}
	\end{align}
	Writing $ W(y,t) = u(y,t) + iv(y,t) $ in \cref{eq:linearizedNLS}, we have the two linearized equations
	\begin{align*}
	u_{t} &= v_{yy} + \omega v - S_{y}^{2}v + \zeta f(A^{2})v + S_{yy}u + 2S_{y}u_{y}\\
	-v_{t} &= u_{yy} + \omega u - S_{y}^{2}u + \zeta f(A^{2})u + 2\zeta f'(A^{2})A^{2}u - S_{yy}v - 2S_{y}v_{y},
	\end{align*}
	which can be collected into the equation
	\begin{align}
	\begin{pmatrix}
	u\\v
	\end{pmatrix}_{t} &= \begin{pmatrix}
	\mathcal{K} & -\mathcal{L}_{-}\\
	\mathcal{L}_{+} & \mathcal{K}
	\end{pmatrix}\begin{pmatrix}
	u\\v
	\end{pmatrix} \label{eq:linearMatrixProblem}
	\end{align}
	where
	\begin{align*}
	\mathcal{K} &= S_{yy} + 2S_{y}\partial_{y}\\
	\mathcal{L}_{+} &= -\omega - \partial_{yy} + S_{y}^{2} - \zeta f(A^{2}) - 2\zeta f'(A^{2})A^{2}\\
	\mathcal{L}_{-} &= -\omega - \partial_{yy} + S_{y}^{2} - \zeta f(A^{2}).
	\end{align*}
	We can further manipulate \cref{eq:linearMatrixProblem}:
	\begin{align}
	\begin{pmatrix}
	u\\v
	\end{pmatrix}_{t} &= \begin{pmatrix}
	0 & -1\\
	1 & 0
	\end{pmatrix}\begin{pmatrix}
	\mathcal{L}_{+} & \mathcal{K}\\
	\mathcal{K}^{T} & \mathcal{L}_{-}
	\end{pmatrix}\begin{pmatrix}
	u\\v
	\end{pmatrix}\\
	&= \mathcal{L}\begin{pmatrix}
	u\\v
	\end{pmatrix}. \label{eq:HamiltonianProblem}
	\end{align}
	Since $ \mathcal{L} $ is the product of a skew-symmetric operator and symmetric operator then \cref{eq:HamiltonianProblem} defines a Hamiltonian eigenvalue problem. From here, the idea is to analyze the Jordan chains of $ \mathcal{L} $. In particular, we wish to find the genericity conditions on the Jordan structure of the linearized operator in terms of the derivatives of $ K $. Having found these conditions, we can then perturb $ \mathcal{L} $ and analyze the affect this has on the generalized kernels. The following result summarizes the genericity conditions.
	\begin{Theorem}[Theorem 1 \cite{LBJM2019}]\label{th:t1}
		Assume that:
		\begin{itemize}
			\item $ (E,\kappa,\omega,\zeta) \in \Omega; $
			\item $ F(x) $ is linearly independent from $ 1,x,\frac{1}{x}; $
			\item $ A(y) $ is non-constant.
		\end{itemize}
		The generalized kernel of $ \mathcal{L} $ (defined in \cref{eq:HamiltonianProblem}) generically takes the form of a direct sum of two Jordan chains of length two. Making the definitions
		\begin{align*}
		\sigma &\coloneqq \{T_{*},\eta\}_{E,\kappa} = T_{*E}\eta_{\kappa} - T_{*\kappa}\eta_{E}\\
		D &\coloneqq \begin{vmatrix}
		K_{\kappa\kappa} & K_{\kappa E} & K_{\kappa\omega} & T_{*}\\
		K_{\kappa E} & K_{EE} & K_{E\omega} & 0\\
		K_{\kappa\omega} & K_{E\omega} & K_{\omega\omega} & 0\\
		T_{*} & 0 & 0 & -M
		\end{vmatrix},
		\end{align*}
		then the genericity conditions on these Jordan chains are:
		\begin{align}
		\sigma &\neq 0\\
		D &\neq 0.
		\end{align}
	\end{Theorem}
	\begin{remark}
		We have chosen to omit several results from \Cref{th:t1} in \cite{LBJM2019}, such as the explicit calculation of the Jordan chains of $ \mathcal{L} $. This is because the Whitham theory in the later sections of this paper makes use of just the genericity conditions on these chains.
	\end{remark}
	\begin{remark}
		In the proof of \Cref{th:t1}, the authors calculate the determinant $ D $ in terms of $ T_{*},M,\eta $ and the Poisson bracket quantities defined in \cref{appendix:app1}. Using the Dodgson-Jacobi-Desnanot condensation identity, they calculate:
		\begin{align*}
		-\frac{\sigma^{3}}{4}D &= \begin{vmatrix}
		a_{2} & b_{2}\\
		b_{2} & d_{2}
		\end{vmatrix},
		\end{align*}
		where $ a_{2},b_{2},d_{2} $ are defined in \cref{appendix:app1}.
	\end{remark}
	We now examine the breakup of the Jordan chains of $ \mathcal{L} $ under perturbations of the quasi-periodic boundary conditions. We start with the following proposition:
	\begin{proposition}[Proposition 1 \cite{LBJM2019}]\label{prop:p1}
		Let $ \mathcal{L} $ be an operator with compact resolvent. Suppose further that $ \mathcal{L} $ has a $ d $-dimensional kernel spanned by $ \{\mathbf{u}_{2j}\}_{j=0}^{d-1} $ and a $ d $-dimensional first generalized kernel spanned by $ \{\mathbf{u}_{2j+1}\}_{j=0}^{d-1} $ satisfying $ \mathcal{L}\mathbf{u}_{2j+1} = \mathbf{u}_{2j}$. Suppose similarly that $ \mathcal{L} $ has a left basis satisfying $ \mathbf{v}_{2j+1}\mathcal{L} = 0,\;\mathbf{v}_{2j}\mathcal{L} = \mathbf{v}_{2j+1} $. Consider a perturbation of the form $ \mathcal{L} + \mu\mathcal{L}^{(1)} + \mu^{2}\mathcal{L}^{(2)} $, where $ \mathcal{L}^{(1)}(\lambda - \mathcal{L})^{-1}$, $\mathcal{L}^{(1)}(\lambda - \mathcal{L}^{(1)})^{-1}\mathcal{L}^{(1)} $ and $ \mathcal{L}^{(2)} $ are bounded operators. Suppose that the first order perturbation satisfies the conditions:
		\begin{align*}
		\mathbf{v}_{2j+1}\mathcal{L}^{(1)}\mathbf{u}_{2k} = 0 \quad \forall j,k = 0,\dots,d-1.
		\end{align*}
		Then, to leading order in $ \mu $, the $ 2d $-dimensional generalized kernel breaks up into $ 2d $ eigenspaces, with the eigenvalues given by
		\begin{align*}
		\lambda(\mu) = \lambda_{1}\mu + O(\mu^{2}),
		\end{align*}
		where $ \lambda_{1} $ is a root of the polynomial
		\begin{align*}
		\det(\lambda_{1}^{2}\mathbf{M}^{(2)} + \lambda_{1}\mathbf{M}^{(1)} + \mathbf{M}^{(0)}) = 0,
		\end{align*}
		with the $ d\times d $ matrices $ \mathbf{M}^{(i)} $ are defined as:
		\begin{align*}
		\mathbf{M}_{j,k}^{(2)} &= \mathbf{v}_{2j+1}\mathbf{u}_{2k+1}\\
		\mathbf{M}_{j,k}^{(1)} &= -\mathbf{v}_{2j}\mathcal{L}^{(1)}\mathbf{u}_{2k} - \mathbf{v}_{2j+1}\mathcal{L}^{(1)}\mathbf{u}_{2k+1}\\
		\mathbf{M}_{j,k}^{(0)} &= \mathbf{v}_{2j+1}\mathcal{L}^{(1)}\mathcal{L}^{-1}\mathcal{L}^{(1)}\mathbf{u}_{2k} - \mathbf{v}_{2j+1}\mathcal{L}^{(2)}\mathbf{u}_{2k}.
		\end{align*}
	\end{proposition}
	\Cref{prop:p1} allows us to relate the eigenvalues $ \lambda $ near the origin to the break up of the generalized kernels of $ \mathcal{L} $ under small perturbations. We will first deal with the case of longitudinal perturbations, that is, the 1D NLS equation \cref{eq:1DNLS}. In particular, when we substitute
	\begin{align*}
	\begin{pmatrix}
	u\\v
	\end{pmatrix} = e^{\lambda t}\mathbf{u}(y)
	\end{align*}
	into the linearized problem \cref{eq:HamiltonianProblem} we have:
	\begin{align}
	\mathcal{L}\mathbf{u} = \lambda \mathbf{u}.\label{eq:spectralProblem}
	\end{align}
	\Cref{eq:spectralProblem} can be written as a system of four first order ODEs in $ y $ with $ T_{*} $-periodic coefficients, and so by applying Floquet theory we see that the spectrum of $ \mathcal{L} $ is the union over all $ \mu \in (-\frac{\pi}{T_{*}},\frac{\pi}{T_{*}}] $ of the spectrum of $ \mathcal{L} $ with quasi-periodic boundary conditions:
	\begin{align*}
	\mathbf{u}(T_{*}) &= e^{i\mu T_{*}}\mathbf{u}(0)\\
	\mathbf{u}_{y}(T_{*}) &= e^{i\mu T_{*}}\mathbf{u}_{y}(0).
	\end{align*}
	If instead we write $ \mathbf{u} = e^{i\mu y}\mathbf{v} $, then $ \mathbf{v} $ has periodic boundary conditions, and the $ \mu $-dependency is instead captured by a $ \mu $-dependent operator:
	\begin{align*}
	\mathcal{L}(\mu)\mathbf{v} &= \lambda\mathbf{v}\\
	\mathcal{L}(\mu) &\coloneqq \mathcal{L} + 2i\mu\begin{pmatrix}
	S_{y} & \partial_{y}\\
	-\partial_{y} & S_{y}
	\end{pmatrix} + \mu^{2}\begin{pmatrix}
	0 & -1\\
	1 & 0
	\end{pmatrix}.
	\end{align*}
	Taking 
	\begin{align*}
	\mathcal{L}^{(1)} &= 2i\begin{pmatrix}
	S_{y} & \partial_{y}\\
	-\partial_{y} & S_{y}
	\end{pmatrix}\\
	\mathcal{L}^{(2)} &= \begin{pmatrix}
	0 & -1\\
	1 & 0
	\end{pmatrix}
	\end{align*}
	allows us to apply \cref{prop:p1}, leading to the following theorem.
	\begin{Theorem}[Corollary 1 \cite{LBJM2019}]\label{th:t2}
		Suppose that $ \mathcal{L} $ is defined as in \cref{eq:HamiltonianProblem}, and further suppose that the genericity conditions of \Cref{th:t1} apply. For small values of the Floquet exponent $ \mu $, the normal form for the four continuous bands of spectrum emerging from the origin in the spectral plane is:
		\begin{align}
		\det\left(\lambda^{2}\begin{pmatrix}
		a_{2} & b_{2}\\
		b_{2} & d_{2}
		\end{pmatrix} + \lambda\mu\begin{pmatrix}
		a_{1} & b_{1}\\
		b_{1} & d_{1}
		\end{pmatrix} + \mu^{2}\begin{pmatrix}
		a_{0} & b_{0}\\
		b_{0} & d_{0}
		\end{pmatrix}\right) + O(5) = 0, \label{eq:lnormalform}
		\end{align}
		where $ O(5) $ consists of terms $ \lambda^{i}\mu^{j} $ with $ i,j > 0 $ and $ i + j \geq 5 $. The matrix entries are given in \cref{appendix:app1}.
	\end{Theorem}
	\begin{corollary}[Modulational instability criterion]\label{cor:c1}
		If any of the roots $ \lambda $ of \cref{eq:lnormalform} are not purely imaginary, then the travelling wave $ \phi $ about which we have linearized is modulationally unstable.
	\end{corollary}
	\section{Whitham modulation theory}\label{sec:s31}
	Our goal in this section is to provide a formal Whitham theory calculation which results in the same criterion for modulational instability as the rigorous results in the previous section, in particular \Cref{th:t2} and \Cref{cor:c1}. We want to show that the equation for the characteristics of the Whitham modulation equations is equivalent to the normal form \cref{eq:lnormalform}. There are several ways to derive the Whitham modulation equations; we choose to use the averaged Lagrangian formulation and the variational principle from Whitham's work \cite{Whi1965B,Whi1970,Whi1999}, since the averaged Lagrangian closely resembles the classical action $ K $ (cf. \cref{eq:classicalAction}) from the rigorous linear theory.\par
	As a side note, many of the symbols of the previous section will be re-used but with the subscript $ * $. These subscripted symbols are seen as completely new, and not connected to their un-subscripted counterparts. We start with the Lagrangian of \cref{eq:1DNLS}:
	\begin{align}
	L(\psi_{t},\psi_{x},\psi,\overline{\psi}_{t},\overline{\psi}_{x},\overline{\psi}) = i(\overline{\psi}\psi_{t} - \psi\overline{\psi}_{t}) + 2\lvert\psi_{x}\rvert^{2} - 2\zeta F(\lvert\psi\rvert^{2}), \label{eq:nlsLagrangian}
	\end{align}
	which satisfies the Euler-Lagrange equations:
	\begin{align*}
	L_{\psi} - \partial_{t}L_{\psi_{t}} - \partial_{x}L_{\psi_{x}} &= 0\\
	L_{\overline{\psi}} - \partial_{t}L_{\overline{\psi}_{t}} - \partial_{x}L_{\overline{\psi}_{x}} &= 0.
	\end{align*}
	We seek solutions to \cref{eq:1DNLS} of the form:
	\begin{align*}
	\psi(x,t) = r(x,t)e^{i\varphi(x,t)}.
	\end{align*}
	Upon subsitution, we have (for imaginary and real parts respectively)
	\begin{align}
	r_{t} &= 2\varphi_{x}r_{x}+\varphi_{xx}r \label{eq:imParts}\\
	-\varphi_{t}r &= r_{xx} - \varphi_{x}^{2}r + \zeta f(r^{2})r, \label{eq:reParts}
	\end{align}
	with the Lagrangian:
	\begin{align}
	L(r_{x},r,\varphi_{t},\varphi_{x}) = -2\varphi_{t}r^{2} + 2(r_{x})^{2} + 2(\varphi_{x})^{2}r^{2} - 2\zeta F(r^{2}) \label{eq:polarLagrangianWhitham}.
	\end{align}
	We now seek (quasi)periodic solutions $ r,\varphi $, with period normalized to $ 2\pi $. Introducing the variable $ \theta = kx - \omega_{*}t $, we let
	\begin{align*}
	r(x,t) &= R(\theta)\\
	\varphi(x,t) &= \beta x - \gamma_{*} t + \Phi(\theta).
	\end{align*}
	The pseudo-phase component $ \beta x - \gamma_{*} t $ is a necessary generalization since only the derivatives of $ \varphi $ appear in $ L $ (cf. \cref{eq:polarLagrangianWhitham}). Consequently, the pseudo-phase ensures that the quantities $ \varphi_{x} $ and $ \varphi_{t} $ have mean $ \beta $ and $ -\gamma_{*} $ respectively when averaged over one period. This is reminiscent of the quasi-periodic boundary conditions set for $ \psi $ in \cref{eq:qpL}. Substituting into \cref{eq:imParts,eq:reParts,eq:polarLagrangianWhitham} yields:
	\begin{align}
	&-\omega_{*} R_{\theta} = 2kR_{\theta}(\beta + k\Phi_{\theta}) + k^{2}\Phi_{\theta\theta}R \label{eq:NLSWhitham1}\\
	&R(\gamma_{*} + \omega_{*} \Phi_{\theta}) = k^{2}R_{\theta\theta} - R(\beta + k\Phi_{\theta})^{2} + \zeta f(R^{2})R \label{eq:NLSWhitham2}\\
	&L(kR_{\theta},R,-\gamma - \omega_{*}\Phi_{\theta},\beta + k\Phi_{\theta}) = 2R^{2}(\gamma_{*} + \omega_{*}\Phi_{\theta}) + 2k^{2}R_{\theta}^{2} + 2R^{2}(\beta + k\Phi_{\theta})^{2} - 2\zeta F(R^{2}). \label{eq:WhithamLagrangian}
	\end{align}
	For a slow modulation, we consider the the parameters $ \omega_{*},\; k,\;\gamma_{*},\;\beta $ to be functions of $ X = \epsilon x$ and $ T = \epsilon t $. We write:
	\begin{align}
	\theta = \epsilon^{-1}\Theta(X,T),\quad \beta x - \gamma_{*} t = \widetilde{\theta} = \epsilon^{-1}\widetilde{\Theta}(X,T), \label{eq:twotiming1}
	\end{align}
	and define:
	\begin{align}
	-\omega_{*}(X,T) = \Theta_{T},\quad k(X,T) = \Theta_{X},\quad -\gamma_{*}(X,T) = \widetilde{\Theta}_{T},\quad \beta(X,T) = \widetilde{\Theta}_{X}.
	\end{align}
	Our functions $ R $ and $ \Phi $ are also written in these variables as:
	\begin{align*}
	R = R(\theta,X,T;\epsilon),\quad \Phi = \Phi(\theta,X,T;\epsilon).
	\end{align*}
	The fact that these functions are evolving on both slow and fast scales is called two-timing, and is closely examined in \cite{Whi1970,Whi1999}. When averaging the Lagrangian, we consider $ R $ and $ \Phi $ to be functions of three independent variables $ \theta,\;X,\;T $, even though $ \theta $ is a function of $ X $ and $ T $ in \cref{eq:twotiming1}. This extra flexibility ensures that secular terms in the asymptotic expansions are suppressed, and in conjunction with the variational principle this is equivalent to a WKB approximation \cite{Whi1999}. According to Whitham's theory \cite{Whi1999,Whi1970}, the leading order modulation equations can be derived from the variational principle:
	\begin{align}
	\delta V(R_{\theta},R,\Phi_{\theta},\Theta,\widetilde{\Theta}) = \delta \int\int\frac{1}{2\pi}\int_{0}^{2\pi}L(kR_{\theta},R,-\gamma_{*} - \omega_{*}\Phi_{\theta},\beta + k\Phi_{\theta})d\theta dXdT = 0. \label{eq:varprinciple}
	\end{align}
	For functions $ h $ which vanish on the $ (\theta,X,T) $ boundary, \cref{eq:varprinciple} means that, for variations in $ R $ we have:
	\begin{align*}
	V(R_{\theta} + h_{\theta},R + h,\Phi_{\theta},\Theta,\widetilde{\Theta}) - V(R_{\theta},R,\Phi_{\theta},\Theta,\widetilde{\Theta}) &= 0\\
	\implies \int\int\frac{1}{2\pi}\int_{0}^{2\pi}h_{\theta}L_{R_{\theta}} + hL_{R} + O(h^{2})d\theta dXdT &= 0.\\
	\end{align*}
	Integrating by parts, we end up with
	\begin{align*}
	\int\int\frac{1}{2\pi}\int_{0}^{2\pi}h(L_{R} - \partial_{\theta}L_{R_{\theta}})d\theta dXdT &= 0.
	\end{align*}
	Since this is true for all such $ h $, we conclude that:
	\begin{align}
	\partial_{\theta}L_{R_{\theta}} - L_{R} = 0. \label{eq:varR}
	\end{align}
	\Cref{eq:varR} is simply \cref{eq:NLSWhitham2}, one of the Euler-Lagrange equations for $ L $ (cf. \cref{eq:WhithamLagrangian}). The other Euler-Lagrange equation comes from considering variations in $ \Phi $, which yields:
	\begin{align}
	\partial_{\theta}L_{\Phi_{\theta}} = 0. \label{eq:varPhi}
	\end{align}
	\Cref{eq:varPhi} has an immediate first integral:
	\begin{align}
	B(X,T) = L_{\Phi_{\theta}}, \label{eq:ELB}
	\end{align}
	so $ B $ is constant with respect to $ \theta $, however it is now added to the ensemble of slowly-varying parameters. \Cref{eq:varPhi} is in fact a multiple of \cref{eq:NLSWhitham1} (the factor is $ 4R $), so finding $ B $ is equivalent to finding a first integral of \cref{eq:NLSWhitham1}. \Cref{eq:varR} also admits an integral, since this equation is an ODE in the variable $ \theta $. To see this, we multiply \cref{eq:varR} by $ R_{\theta} $:
	\begin{align*}
	R_{\theta}\partial_{\theta}L_{R_{\theta}} - R_{\theta}L_{R} &= 0\\
	\implies \partial_{\theta}(R_{\theta}L_{R_{\theta}}) - R_{\theta\theta}L_{R_{\theta}} - R_{\theta}L_{R} &= 0\\
	\implies \partial_{\theta}(R_{\theta}L_{R_{\theta}}) - \partial_{\theta}L + B\Phi_{\theta\theta} &= 0.
	\end{align*}
	Note that we have used \cref{eq:varPhi} to write the $ \theta $-derivative of $ L $ as:
	\begin{align*}
	\partial_{\theta}L &= R_{\theta\theta}L_{R_{\theta}} + R_{\theta}L_{R} + B\Phi_{\theta\theta}.
	\end{align*}
	Integrating, we have:
	\begin{align}
	R_{\theta}L_{R_{\theta}} + B\Phi_{\theta} - L &= A_{*}(X,T), \label{eq:ELA}
	\end{align}
	for a slowly-varying parameter $ A_{*} $. For variations in $ \Theta $, we consider $ \Theta(X,T) + h(X,T) $ for an appropriate function $ h $ vanishing on the $ (X,T) $ boundary. This means that $ k $ will be replaced by $ k + h_{X} $, and similarly $ \omega_{*} $ by $ \omega - h_{T} $.
	Writing the averaged Lagrangian as
	\begin{align}
	\overline{L}(\omega_{*},k,\gamma_{*},\beta) = \frac{1}{2\pi}\int_{0}^{2\pi}Ld\theta, \label{eq:avgLagrange}
	\end{align}
	we have
	\begin{align*}
	V(R_{\theta},R,\Phi_{\theta},\Theta + h,\widetilde{\Theta}) - V(R_{\theta},R,\Phi_{\theta},\Theta,\widetilde{\Theta}) &= 0\\
	\implies \int\int\overline{L}(\omega_{*} - h_{T},k + h_{X},\gamma_{*},\beta) - \overline{L}(\omega_{*},k,\gamma_{*},\beta) dXdT &= 0,
	\end{align*}
	and hence:
	\begin{align}
	\int\int h_{X}\overline{L}_{k} - h_{T}\overline{L}_{\omega_{*}} + O(h^{2})dXdT &= 0. \label{eq:int1}
	\end{align}
	From \cref{eq:int1}, we have:
	\begin{align*}
	\int\int h(\partial_{T}\overline{L}_{\omega_{*}} - \partial_{X}\overline{L}_{k})dXdT &= 0,
	\end{align*}
	which is true for all such $ h $, and so we conclude that:
	\begin{align}
	\partial_{T}\overline{L}_{\omega_{*}} - \partial_{X}\overline{L}_{k} = 0.\label{eq:mod1}
	\end{align}
	\Cref{eq:mod1} is called a modulation equation since the quantities involved are functions of the slowly-modulated variables $ X $ and $ T $. Similarly for variations in $ \widetilde{\Theta} $, we have another modulation equation:
	\begin{align}
	\partial_{T}\overline{L}_{\gamma_{*}} - \partial_{X}\overline{L}_{\beta} = 0.
	\end{align}
	The last two modulation equations are associated with the conservation of waves:
	\begin{align}
	\partial_{T}k + \partial_{X}\omega_{*} &= 0 \label{eq:mod3}\\
	\partial_{T}\beta + \partial_{X}\gamma_{*} &= 0 \label{eq:mod4},
	\end{align}
	which is a consequence of requiring the mixed partial derivatives of $ \Theta(X,T) $ and $ \widetilde{\Theta}(X,T) $ to be equal. At this point, we have derived four modulation equations in the four parameters $ (\omega_{*},k,\gamma_{*},\beta) $. In order to evaluate the averaged Lagrangian \cref{eq:avgLagrange}, we need to solve \cref{eq:ELB,eq:ELA} for $ R,\Phi $. Instead, Whitham proposes using the first integrals $ A_{*}(X,T) $ and $ B_{*}(X,T) $ to simplify the form of the averaged Lagrangian \cite{Whi1999}. We can avoid directly computing $ R,\Phi $, and moreover we can absorb the dispersion relations into one variational principle. To this end, we consider a Hamiltonian transformation with the variables:
	\begin{align}
	\Pi_{1} = L_{R_{\theta}},\quad \Pi_{2} = L_{\Phi_{\theta}} = B. \label{eq:momenta}
	\end{align}
	These are generalized momenta, and so we apply a Legendre transform to $ L $ in order to eliminate $ R_{\theta} $ and $ \Phi_{\theta} $ from the equations:
	\begin{align}
	H(\Pi_{1},\Pi_{2},R,\Phi;\omega_{*},k,\gamma_{*},\beta) = R_{\theta}\Pi_{1} + \Phi_{\theta}\Pi_{2} - L. \label{eq:WhithamHamiltonian}
	\end{align}
	We see from the transformation that
	\begin{align}
	R_{\theta} = \partial_{\Pi_{1}}H,\quad \Phi_{\theta} = \partial_{\Pi_{2}}H, \label{eq:H1}
	\end{align}
	and it follows from \cref{eq:varR} and \cref{eq:varPhi} respectively that
	\begin{align}
	\partial_{\theta}\Pi_{1} = -\partial_{R}H,\quad \partial_{\theta}\Pi_{2} = -\partial_{\Phi}H = 0. \label{eq:H2}
	\end{align}
	\Cref{eq:H1,eq:H2} are Hamilton's equations. We can now write the earlier variational principle \cref{eq:varprinciple} as:
	\begin{align}
	\delta\int\int\overline{L}dXdT &= 0, \label{eq:avgvarprinciple}
	\end{align}
	where
	\begin{align}
	\overline{L} &= \frac{1}{2\pi}\int_{0}^{2\pi}(R_{\theta}\Pi_{1} + B\Phi_{\theta} - H)d\theta \label{eq:avgL}\\
	&= \frac{1}{2\pi}\int_{0}^{2\pi}(R_{\theta}\Pi_{1}- H)d\theta,
	\end{align}
	since $ \Phi $ is $ 2\pi $-periodic. The averaged Lagrangian is now a function of $ H,\; B $ and the previous four parameters. We note that Hamilton's \cref{eq:H1,eq:H2} follow from the independent variations of $ \Pi_{1},R,\Pi_{2},\Phi $ in \cref{eq:avgL}. We use this extension as Whitham describes in \cite{Whi1970,Whi1999}. Next, we observe that \cref{eq:WhithamHamiltonian} is the same as \cref{eq:ELA}, so we identify:
	\begin{align}
	H(\Pi_{1},\Pi_{2},R;\omega_{*},k,\gamma_{*},\beta) = A_{*}(X,T). \label{eq:HA}
	\end{align}
	The stationary values of \cref{eq:avgL} also satisfy \cref{eq:HA}, meaning that we can restrict the stationary values of \cref{eq:avgL} to the class of functions $ R,\;\Phi,\;\Pi_{1},\;\Pi_{2} $ which satisfy \cref{eq:HA}. Importantly, this is the only restriction we make; using the dispersion relation or any information about the forms of the solutions $ \Pi_{1} $ and $ R_{\theta} $ (\cref{eq:H1}) would result in $ \overline{L} $ having no variation. We relabel $ H(\Pi_{1},\Pi_{2},R,\Phi;\omega_{*},k,\gamma_{*},\beta) $ as $ H(X,T) $ in \cref{eq:avgL}, and we solve \cref{eq:HA} for $ \Pi_{1}(R;H,\omega_{*},k,B,\gamma_{*},\beta) $ which yields:
	\begin{align}
	\mathscr{L}(H,\omega_{*},k,B,\gamma_{*},\beta) = \frac{1}{2\pi}\oint\Pi_{1}dR - H,\label{eq:averageLagrangian}
	\end{align}
	where the integral is taken around the orbit of $ R $. The variational principle \cref{eq:avgvarprinciple} can now be written as
	\begin{align}
	\delta\int\int\mathscr{L}(H,\omega_{*},k,B,\gamma_{*},\beta)dXdT = 0. \label{eq:modVarPrinciple}
	\end{align}
	As mentioned earlier, we have now added the two parameters $ H $ and $ B $, however we have exchanged the variational equations \cref{eq:H1,eq:H2} for variations in $ H $ and $ B $:
	\begin{align}
	\mathscr{L}_{H} &= 0 \label{eq:varH}\\
	\mathscr{L}_{B} &= 0 \label{eq:varB}.
	\end{align}
	\Cref{eq:varH,eq:varB} are relations between the parameters of the periodic wavetrain, and so they are in fact the dispersion relations. Our aim is to use these equations in order to eliminate two of the six parameters from the four modulation equations, resulting in a homogeneous system of first-order PDEs. Using \cref{eq:modVarPrinciple,eq:mod3,eq:mod4}, we now have a complete picture of the Whitham theory:
	\begin{align}
	\partial_{T}\mathscr{L}_{\omega_{*}} - \partial_{X}\mathscr{L}_{k} &= 0 \label{eq:modwk}\\
	\partial_{T}\mathscr{L}_{\gamma_{*}} - \partial_{X}\mathscr{L}_{\beta} &= 0 \label{eq:modgb}\\
	\partial_{T}k + \partial_{X}\omega_{*} &= 0 \label{eq:modc1}\\
	\partial_{T}\beta + \partial_{X}\gamma_{*} &= 0. \label{eq:modc2}
	\end{align}
	\subsection{Whitham theory applied to NLS}\label{sec:s32}
	In this subsection, we apply the more general theory and observations from \cref{sec:s31} to \cref{eq:1DNLS}. We provide a direct computation of the characteristics of the Whitham modulation equations, from which we conclude that the modulational instability criterion from the Whitham theory agrees with the spectral analysis in \cref{sec:s2}. For the sake of brevity, we include more detailed calculations in \cref{appendix:app2,appendix:app3,appendix:app4,appendix:app5,appendix:app6}.\par
	We calculate $ \mathscr{L} $ in terms of the parameters $ H,\omega_{*},k,B,\gamma_{*}\beta $ using \cref{eq:NLSWhitham1,eq:NLSWhitham2,eq:WhithamLagrangian}. Firstly, $ \Pi_{2} = B $ amounts to taking an integral of \cref{eq:NLSWhitham1}, which yields:
	\begin{align}
	B = 2\omega_{*}R^{2} + 4kR^{2}(\beta + k\Phi_{\theta}). \label{eq:pi2B}
	\end{align}
	Next, we calculate $ H $ using \cref{eq:WhithamHamiltonian}, eliminating $ \Phi_{\theta} $ and $ R_{\theta} $ via \cref{eq:pi2B} and $ \Pi_{1} = L_{R_{\theta}} = 4k^{2}R_{\theta} $:
	\begin{align*}
	H &= \frac{\Pi_{1}^{2}}{4k^{2}} + B\Phi_{\theta} - \frac{\Pi_{1}^{2}}{8k^{2}} - 2R^{2}(\gamma_{*} + \omega_{*}\Phi_{\theta}) - 2R^{2}(\beta + k\Phi_{\theta})^{2} + 2\zeta F(R^{2})\\
	&= \frac{\Pi_{1}^{2}}{8k^{2}} - \frac{\beta B}{k} - \frac{\omega_{*} B}{2k^{2}} + 2R^{2}\left(\frac{\beta\omega_{*}}{k} + \frac{\omega_{*}^{2}}{4k^{2}} - \gamma_{*}\right) + \frac{B^{2}}{8k^{2}R^{2}} + 2\zeta F(R^{2})\\
	\implies \frac{\Pi_{1}^{2}}{8k^{2}} &= H + \frac{\beta B}{k} + \frac{\omega_{*} B}{2k^{2}} - 2R^{2}\left(\frac{\beta\omega_{*}}{k} + \frac{\omega_{*}^{2}}{4k^{2}} - \gamma_{*}\right) - \frac{B^{2}}{8k^{2}R^{2}} - 2\zeta F(R^{2}).
	\end{align*}
	This can also be written as:
	\begin{align}
	2k^{2}R_{\theta}^{2} = H + \frac{\beta B}{k} + \frac{\omega_{*} B}{2k^{2}} - 2R^{2}\left(\frac{\beta\omega_{*}}{k} + \frac{\omega_{*}^{2}}{4k^{2}} - \gamma_{*}\right) - \frac{B^{2}}{8k^{2}R^{2}} - 2\zeta F(R^{2}), \label{eq:NLSRode}
	\end{align}
	which we identify as the first integral of \cref{eq:NLSWhitham2} once $ \Phi_{\theta} $ has been eliminated, i.e:
	\begin{align}
	k^{2}R_{\theta\theta} = -R\left(\frac{\beta\omega_{*}}{k} + \frac{\omega_{*}^{2}}{4k^{2}} - \gamma_{*}\right) + \frac{B^{2}}{16k^{2}R^{3}} - \zeta f(R^{2})R. \label{eq:WhithamNLS2ndorder}
	\end{align}
	The averaged Lagrangian is hence:
	\begin{align}
	\mathscr{L} = \frac{k\sqrt{2}}{\pi}\oint\sqrt{H + \frac{\beta B}{k} + \frac{\omega_{*} B}{2k^{2}} - 2R^{2}\left(\frac{\beta\omega_{*}}{k} + \frac{\omega_{*}^{2}}{4k^{2}} - \gamma_{*}\right) - \frac{B^{2}}{8k^{2}R^{2}} - 2\zeta F(R^{2})}dR - H. \label{eq:ml}
	\end{align}
	Following Whitham's examples \cite{Whi1965A,Whi1999}, we introduce a function $ W $ which is essentially the classical action:
	\begin{align}
	W(H,\omega_{*},k,B,\gamma_{*},\beta) &= \frac{1}{2\pi}\oint\frac{\Pi_{1}}{2k}dR\\
	&= \frac{\sqrt{2}}{2\pi}\oint\sqrt{H + \frac{\beta B}{k} + \frac{\omega_{*} B}{2k^{2}} - 2R^{2}\left(\frac{\beta\omega_{*}}{k} + \frac{\omega_{*}^{2}}{4k^{2}} - \gamma_{*}\right) - \frac{B^{2}}{8k^{2}R^{2}} - 2\zeta F(R^{2})}dR. \label{eq:W}
	\end{align}
	This allows us to write \cref{eq:ml} as:
	\begin{align}
	\mathscr{L} = 2kW - H. \label{eq:ml1}
	\end{align}
	For a generic evaluation of $ W $ from \cref{eq:W}, we require assumptions similar to those listed in \cref{def:d1}. We derive these assumptions by equating our periodic solution $ \psi(x,t) = R(\theta)\exp(i(\beta x - \gamma_{*} t + \Phi(\theta))) $ with the periodic solution \cref{eq:NLSphi}, which yields (once translation invariance and phase invariance are taken into account):
	\begin{align}
	R(\theta) &= A(y) \label{eq:equating1}\\
	\beta x - \gamma_{*} t + \Phi(\theta) &= \left(\omega - \frac{c^{2}}{4}\right)t + \frac{cy}{2} + S(y) \label{eq:equating2}
	\end{align}
	Note that we have redefined $ \omega $ as $ \omega - \frac{c^{2}}{4} $ as suggested by \cref{eq:reducedpsi}. Taking $ x $ and $ t $ derivatives of \cref{eq:equating1}, we have:
	\begin{align}
	kR_{\theta} &= A_{y},\quad -\omega_{*}R_{\theta} = cA_{y} \label{eq:eqRA}\\
	\implies c &= -\frac{\omega_{*}}{k}. \label{eq:cU}
	\end{align}
	Similarly for \cref{eq:equating2} we have:
	\begin{align}
	\beta + k\Phi_{\theta} &= \frac{c}{2} + S_{y},\quad -\gamma_{*} - \omega_{*}\Phi_{\theta} = \omega - \frac{c^{2}}{4} + \frac{c^{2}}{2} + cS_{y}\label{eq:eqPhiS}\\
	\implies \omega &=\frac{\beta\omega_{*}}{k} + \frac{\omega_{*}^{2}}{4k^{2}} - \gamma_{*}. \label{eq:omega}
	\end{align}
	Substituting \cref{eq:eqRA,eq:cU,eq:eqPhiS} into \cref{eq:NLSSode} yields:
	\begin{align*}
	\beta + \frac{\omega_{*}}{2k} + k\Phi_{\theta} &= \frac{\kappa}{R^{2}},
	\end{align*}
	and recalling \cref{eq:pi2B} we identify:
	\begin{align}
	\kappa = \frac{B}{4k}. \label{eq:kappa}
	\end{align}
	Finally, using \cref{eq:equating1,eq:eqRA,eq:cU,eq:omega,eq:kappa}, we see that \cref{eq:NLSRode} is twice \cref{eq:NLSode}, with:
	\begin{align}
	E &= \frac{1}{4}\left(H + \frac{\beta B}{k} + \frac{\omega_{*}B}{2k^{2}}\right). \label{eq:E}
	\end{align}
	\Cref{eq:cU,eq:omega,eq:kappa,eq:E} now express the parameters of the linear theory in terms of our Whitham parameters. Recalling \cref{def:d1}, we have that the parameters $ (H,\omega_{*},k,B,\gamma_{*},\beta) $ exist such that:
	\begin{itemize}
		\item $ \frac{B}{k} > 0 $;\\
		\item The function $ P(R) = H + \frac{\beta B}{k} + \frac{\omega_{*} B}{2k^{2}} - 2R^{2}\big(\frac{\beta\omega_{*}}{k} + \frac{\omega_{*}^{2}}{4k^{2}} - \gamma_{*}\big) - \frac{B^{2}}{8k^{2}R^{2}} - 2\zeta F(R^{2}) $ has two positive, real, simple roots $ R_{-} = a_{-} $ and $ R_{+} = a_{+} $, with $ R_{-} < R_{+} $ and $ P(R)  $ real and positive for $ R \in (R_{-},R_{+}) $.
	\end{itemize}
	We make the change of variables:
	\begin{align*}
	U &= \frac{\omega_{*}}{k}\\
	J &= \frac{B}{k}.
	\end{align*}
	This eliminates the explicit dependence of $ P(R) $ on $ k $, which will be advantageous when we calculate the characteristics of the modulation equations. Also note that the two roots $ R_{-},\;R_{+} $ are now functions of $ H,U,J,\gamma_{*},\beta,\zeta $. Updating $ W $ in \cref{eq:W}, we have:
	\begin{align}
	W(H,U,J,\gamma_{*},\beta,\zeta) = \frac{\sqrt{2}}{\pi}\int_{R_{-}}^{R_{+}}\sqrt{H + \beta J + \frac{UJ}{2} - 2R^{2}\left(\beta U + \frac{U^{2}}{4} - \gamma_{*}\right) - \frac{J^{2}}{8R^{2}} - 2\zeta F(R^{2})}dR. \label{eq:finalW}
	\end{align}
	Moreover, we can relate the classical action from \cref{eq:classicalAction} with $ W $:
	\begin{align}
	K = \pi W.\label{eq:KW}
	\end{align}
	We now calculate the derivatives of $ W $ with respect to the parameters:
	\begin{align}
	W_{H} &= \frac{\sqrt{2}}{\pi}\int_{R_{-}}^{R_{+}}\frac{1}{2\sqrt{H + \beta J + \frac{UJ}{2} - 2R^{2}\left(\beta U + \frac{U^{2}}{4} - \gamma_{*}\right) - \frac{J^{2}}{8R^{2}} - 2\zeta F(R^{2})}}dR \label{eq:WH}\\
	W_{J} &= W_{H}\left(\beta + \frac{U}{2}\right) - \eta_{*} \label{eq:WJ}\\
	\eta_{*} &\coloneqq \frac{\sqrt{2}}{\pi}\int_{R_{-}}^{R_{+}}\frac{J}{8R^{2}\sqrt{H + \beta J + \frac{UJ}{2} - 2R^{2}\left(\beta U + \frac{U^{2}}{4} - \gamma_{*}\right) - \frac{J^{2}}{8R^{2}} - 2\zeta F(R^{2})}}dR \label{eq:eta}\\
	M_{*} &\coloneqq W_{\gamma_{*}} = \frac{\sqrt{2}}{\pi}\int_{R_{-}}^{R_{+}}\frac{R^{2}}{\sqrt{H + \beta J + \frac{UJ}{2} - 2R^{2}\left(\beta U + \frac{U^{2}}{4} - \gamma_{*}\right) - \frac{J^{2}}{8R^{2}} - 2\zeta F(R^{2})}}dR \label{eq:M}\\
	W_{U} &= \frac{J}{2}W_{H} - M_{*}\left(\beta + \frac{U}{2}\right) \label{eq:WU}\\
	W_{\beta} &= JW_{H} - UM_{*} \label{eq:Wbeta}\\
	W_{\zeta} &= -\frac{\sqrt{2}}{\pi}\int_{R_{-}}^{R_{+}}\frac{F(R^{2})}{\sqrt{H + \beta J + \frac{UJ}{2} - 2R^{2}\left(\beta U + \frac{U^{2}}{4} - \gamma_{*}\right) - \frac{J^{2}}{8R^{2}} - 2\zeta F(R^{2})}}dR. \label{eq:Wzeta}
	\end{align}
	We choose the notation $ \eta_{*} $ and $ M_{*} $ because these integrals are closely related to the $ \eta $ and $ M $ from the linear theory. We give the relations in \cref{appendix:app2}. Using \cref{eq:ml1}, we have:
	\begin{align*}
	\mathscr{L}_{\omega_{*}} &= 2W_{U}\\
	\mathscr{L}_{k} &= 2W - 2UW_{U} - 2JW_{J}\\
	\mathscr{L}_{\gamma_{*}} &= 2kW_{\gamma_{*}}\\
	\mathscr{L}_{\beta} &= 2kW_{\beta}
	\end{align*}
	We are now in a position to write the modulation equations in terms of the derivatives of $ W $. For \cref{eq:modwk,eq:modgb} we have:
	\begin{align}
	\partial_{T}W_{U} + U\partial_{X}W_{U} + J\partial_{X}W_{J} + W_{U}U_{X} + W_{J}J_{X} - \partial_{X}W &= 0 \label{eq:newmod1}\\
	\partial_{T}\left(2kW_{\gamma_{*}}\right) - \partial_{X}\left(2kW_{\beta}\right) &= 0. \label{eq:newmod2}
	\end{align}
	We now seek to use the variational equations \cref{eq:varH,eq:varB} to eliminate two parameters from the Whitham system. Using \cref{eq:varH} with \cref{eq:ml1} we have:
	\begin{align}
	\mathscr{L}_{H} &= 0\\
	\implies W_{H} &= \frac{1}{2k}. \label{eq:disprel}
	\end{align}
	Similarly for \cref{eq:varB}:
	\begin{align}
	\mathscr{L}_{B} &= 0\\
	\implies W_{J} &= 0. \label{eq:Bdisprel}
	\end{align}
	Since $ W $ and its derivatives have no explicit dependence on $ k $, then \cref{eq:disprel} is in fact the dispersion relation for $ k $ in terms of the other parameters $ H,U,J,\gamma_{*},\beta,\zeta $. Substituting \cref{eq:disprel} into the modulation equations allows us to eliminate $ k $ entirely. This is particularly relevant for \cref{eq:modc1}:
	\begin{align*}
	\partial_{T}k + \partial_{X}(kU) &= 0\\
	\implies \partial_{T}\left(\frac{1}{2W_{H}}\right) + \partial_{X}\left(\frac{U}{2W_{H}}\right) &= 0,
	\end{align*}
	which simplifies to:
	\begin{align*}
	\partial_{T}W_{H} + U\partial_{X}W_{H} - W_{H}U_{X} = 0.
	\end{align*}
	\Cref{eq:newmod1} is already free of $ k $, but \cref{eq:newmod2} can be simplified using \cref{eq:disprel,eq:Wbeta,eq:M}:
	\begin{align*}
	W_{H}\partial_{T}M_{*} - M_{*}\partial_{T}W_{H} + U	W_{H}\partial_{X}M_{*} - UM_{*}\partial_{X}W_{H} + W_{H}M_{*}U_{X} - W_{H}^{2}J_{X} = 0.
	\end{align*}
	We also observe that \cref{eq:Bdisprel} can eliminate terms in \cref{eq:newmod1}. The four modulation equations are now:
	\begin{align}
	\beta_{T} + \gamma_{*X} &= 0 \label{eq:mod1Final}\\
	\partial_{T}W_{H} + U\partial_{X}W_{H} - W_{H}U_{X} &= 0 \label{eq:mod2Final}\\
	\partial_{T}W_{U} + U\partial_{X}W_{U} + W_{U}U_{X} - \partial_{X}W &= 0 \label{eq:mod3Final}\\
	W_{H}\partial_{T}M_{*} - M_{*}\partial_{T}W_{H} + U	W_{H}\partial_{X}M_{*} - UM_{*}\partial_{X}W_{H} + W_{H}M_{*}U_{X} - W_{H}^{2}J_{X} &= 0 \label{eq:mod4Final}.
	\end{align}
	Recalling the genericity condition $ \sigma \neq 0 $ from \Cref{th:t1}, we have that $ \sigma_{*} = \{W_{H},W_{J}\}_{J,H} $ defined in \cref{appendix:app3} is non-zero, which implies that the derivatives of $ W_{J} $ with respect to the parameters $ H $ and $ J $ do no vanish simultaneously. We first assume that $ W_{HJ} \neq 0 $. By the implicit function theorem, there exists a continuously differentiable function $ g $ defined on the appropriate parameter space such that:
	\begin{align}
	H = g(U,J,\gamma_{*},\beta). \label{eq:gimplicit}
	\end{align}
	Next we take derivatives of the equation $ W_{J} = 0 $:
	\begin{align*}
	\partial_{J}W_{J}(g,U,J,\gamma_{*},\beta) &= 0\\
	\implies g_{J}W_{HJ} + W_{JJ} &= 0
	\end{align*}
	Similarly for the other derivatives:
	\begin{align}
	g_{J}W_{HJ} &= -W_{JJ} \label{eq:gJ}\\
	g_{U}W_{HJ} &= -W_{UJ} \label{eq:gU}\\
	g_{\gamma_{*}}W_{HJ} &= -W_{\gamma_{*}J} \label{eq:ggamma}\\
	g_{\beta}W_{HJ} &= -W_{\beta J}. \label{eq:gbeta}
	\end{align}
	\Cref{eq:gJ,eq:gU,eq:ggamma,eq:gbeta} allow us to expand the $ X $ and $ T $ derivatives of $ W_{H} $ and $ M_{*} $ in \cref{eq:mod2Final,eq:mod3Final,eq:mod4Final} by using the chain rule. For example:
	\begin{align*}
	\partial_{T}W_{H} &= g_{T}W_{HH} + J_{T}W_{HJ} + U_{T}W_{HU} + \gamma_{*T}W_{H\gamma_{*}} + \beta_{T}W_{H\beta}\\
	&= (g_{J}J_{T} + g_{U}U_{T} + g_{\gamma_{*}}\gamma_{*T} + g_{\beta}\beta_{T})W_{HH} + J_{T}W_{HJ} + U_{T}W_{HU} + \gamma_{*T}W_{H\gamma_{*}} + \beta_{T}W_{H\beta}\\
	&= (g_{J}W_{HH} + W_{HJ})J_{T} + (g_{U}W_{HH} + W_{HU})U_{T} + (g_{\gamma_{*}}W_{HH} + W_{H\gamma_{*}})\gamma_{*T} + (g_{\beta}W_{HH} + W_{H\beta})\beta_{T}.
	\end{align*}
	If we multiply the derivative by $ W_{HJ} $, we have:
	\begin{align*}
	W_{HJ}\partial_{T}W_{H} &= (W_{HJ}^{2} - W_{HH}W_{JJ})J_{T} + (W_{HU}W_{HJ} - W_{HH}W_{UJ})U_{T} + (W_{H\gamma_{*}}W_{HJ} - W_{HH}W_{\gamma_{*}J})\gamma_{*T}\\
	&+ (W_{H\beta}W_{HJ} - W_{HH}W_{\beta J})\beta_{T}\\
	&= \{W_{H},W_{J}\}_{J,H}J_{T} + \{W_{H},W_{U}\}_{J,H}U_{T} + \{W_{H},W_{\gamma_{*}}\}_{J,H}\gamma_{*T} + \{W_{H},W_{\beta}\}_{J,H}\beta_{T}.
	\end{align*}
	Using \cref{appendix:app3}, we can identify
	\begin{align*}
	\{W_{H},W_{J}\}_{J,H} &= \sigma_{*}\\
	\{W_{H},W_{\gamma_{*}}\}_{J,H} &= -\{W_{H},W_{J}\}_{H,\gamma_{*}} = -\rho_{*},
	\end{align*}
	where we have used the properties of Poisson brackets quoted in \cref{appendix:app3}.
	\begin{remark}
		An important consideration when computing these Poisson brackets is the order of differentiation and substitution. The correct order is to compute the Poisson bracket and then substitute $ H = g(U,J,\gamma_{*},\beta) $ into the expression, rather than applying the chain rule within a Poisson bracket. This is because the Poisson brackets produce identities that hold regardless of the equations tying the parameters together; we are using them as a convenient notation for leveraging the symmetry of the mixed partial derivatives of $ W $. As an example, we make the substitution $ H = g(U,J,\gamma_{*},\beta)  $ into the identity $ \{W_{H},W_{\gamma_{*}}\}_{J,H} = -\{W_{H},W_{J}\}_{H,\gamma_{*}} = -\rho_{*} $ instead of computing $ \partial_{J}W_{H}(g,U,J,\gamma_{*},\beta) $ and other such derivatives.
	\end{remark}
	For $ \{W_{H},W_{U}\}_{J,H} $, we have:
	\begin{align*}
	\{W_{H},W_{U}\}_{J,H} &= \bigg\{W_{H},\frac{JW_{H}}{2} - \left(\beta + \frac{U}{2}\right)W_{\gamma_{*}}\bigg\}_{J,H}\\
	&= \{W_{H},\frac{JW_{H}}{2}\}_{J,H} - \left(\beta + \frac{U}{2}\right)\{W_{H},W_{\gamma_{*}}\}_{J,H}\\
	&= \nu_{*} + \left(\beta + \frac{U}{2}\right)\rho_{*},
	\end{align*}
	and similarly we calculate
	\begin{align*}
	\{W_{H},W_{\beta}\}_{J,H} = 2\nu_{*} + U\rho_{*},
	\end{align*}
	allowing us to write $ W_{HJ}\partial_{T}W_{H} $ in the following way:
	\begin{align}
	W_{HJ}\partial_{T}W_{H} = \begin{pmatrix}
	\sigma_{*} & -\rho_{*} & \nu_{*} + \left(\beta + \frac{U}{2}\right)\rho_{*} & 2\nu_{*} + U\rho_{*}
	\end{pmatrix}\begin{pmatrix}
	J_{T}\\\gamma_{*T}\\U_{T}\\\beta_{T}
	\end{pmatrix}. \label{eq:vecComp}
	\end{align}
	Repeating the process of expressing the derivatives in the modulation equations as products of vectors as in \cref{eq:vecComp}, we can write the modulation equations as a quasi-linear first order system:
	\begin{align}
	A\begin{pmatrix}
	J_{T}\\\gamma_{*T}\\U_{T}\\\beta_{T}
	\end{pmatrix} + a\begin{pmatrix}
	J_{X}\\\gamma_{*X}\\U_{X}\\\beta_{X}
	\end{pmatrix} = 0 \label{eq:qlsystem}
	\end{align}
	with
	\begin{align}
	A &= \begin{pmatrix}
	0 & 0 & 0 & 1\\
	A_{21} & A_{22} & A_{23} & A_{24}\\
	A_{31} & A_{32} & A_{33} & A_{34}\\
	A_{41} & A_{42} & A_{43} & A_{44}
	\end{pmatrix} \label{eq:matrixA}\\
	a &= \begin{pmatrix}
	0 & 1 & 0 & 0\\
	UA_{21} & UA_{22} & UA_{23} - 2\tau & UA_{24}\\
	UA_{31} + a_{31} & UA_{32} + a_{32} & UA_{33} + a_{33} & UA_{34} + a_{34}\\
	UA_{41} - 2\tau W_{H} & UA_{42} & UA_{43} + 2\tau M & UA_{44}
	\end{pmatrix} \label{eq:matrixa}
	\end{align}
	and the coefficients (calculated in \cref{appendix:app4})
	\begin{align*}
	A_{21} &= \sigma_{*}\\
	A_{22} &= -\rho_{*}\\
	A_{23} &= \nu_{*} + \rho_{*}\left(\beta +  \frac{U}{2}\right)\\
	A_{24} &= 2\nu_{*} + U\rho_{*}\\
	A_{31} &= \tau_{*} + \frac{J\sigma_{*}}{2} + \left(\beta + \frac{U}{2}\right)\Gamma\\
	A_{32} &= -\frac{J}{2}\rho_{*} - \left(\beta + \frac{U}{2}\right)\{W_{H},W_{\gamma_{*}}\}_{J,\gamma_{*}}\\
	A_{33} &= \frac{J}{2}\nu_{*} + \frac{W_{H}M_{H}}{2}\left(\beta + \frac{U}{2}\right) - \frac{M_{*}W_{HJ}}{2} + \frac{J}{2}\rho_{*}\left(\beta + \frac{U}{2}\right) + \left(\beta + \frac{U}{2}\right)^{2}\{W_{H},W_{\gamma_{*}}\}_{J,\gamma_{*}}\\
	A_{34} &= J\nu_{*} + W_{H}M_{*H}\left(\beta + \frac{U}{2}\right) - M_{*}W_{HJ} + \frac{J}{2}U\rho_{*} + U\left(\beta + \frac{U}{2}\right)\{W_{H},W_{\gamma_{*}}\}_{J,\gamma_{*}}\\
	A_{41} &= -\Gamma W_{H} - \sigma_{*}M_{*}\\
	A_{42} &= W_{H}\{W_{H},W_{\gamma_{*}}\}_{J,\gamma_{*}} + \rho_{*}M_{*}\\
	A_{43} &= -\frac{W_{H}^{2}M_{*H}}{2} -W_{H}\left(\beta + \frac{U}{2}\right)\{W_{H},W_{\gamma_{*}}\}_{J,\gamma_{*}} - M_{*}(\nu_{*} + \rho_{*}\left(\beta + \frac{U}{2}\right))\\
	A_{44} &= -W_{H}^{2}M_{*H} - UW_{H}\{W_{H},W_{\gamma_{*}}\}_{J,\gamma_{*}} -M_{*}(2\nu_{*} + U\rho_{*})\\
	a_{31} &= W_{H}W_{JJ}\\
	a_{32} &= W_{H}M_{*J} - W_{HJ}M_{*}\\
	a_{33} &= \frac{W_{H}^{2}}{2} + J\tau_{*} - W_{H}M_{*J}\left(\beta + \frac{U}{2}\right)\\
	a_{34} &= W_{H}^{2} + UM_{*}W_{HJ} - UW_{H}M_{*J}.
	\end{align*}
	If any of the characteristics of the system in \cref{eq:qlsystem} are complex, then according to Whitham's theory \cite{Whi1999}, the system is modulationally unstable. Our task is to find the characteristics of \cref{eq:qlsystem} by solving for $ X',T' $ in:
	\begin{align}
	\det(AX' - aT') = 0, \label{eq:qlcharacteristics}
	\end{align}
	where $ X' = \frac{dX}{ds},\;T' = \frac{dT}{ds} $ along a characteristic curve in the $ (X,T) $-plane parametrized by $ s $. The degenerate case $ X'=T'=0 $ in \cref{eq:qlcharacteristics} corresponds to the non-existence of $ c_{1},c_{2} $ not both zero such that
	\begin{align}
	\det(c_{1}A + c_{2}a) \neq 0. \label{eq:nondegenerate}
	\end{align}
	To demonstrate that our system indeed satisfies \cref{eq:nondegenerate}, consider:
	\begin{align*}
	\det{A} &= -W_{HJ}\tau_{*}\left(M_{*}(\Gamma M_{*H} + \rho_{*}M_{*J} + \sigma_{*}M_{*\gamma_{*}}) - \frac{W_{H}^{2}}{2}\{W_{H},W_{\gamma_{*}}\}_{H,\gamma_{*}}\right)\\
	&= -\frac{W_{HJ}\tau_{*}}{2^{9}\pi^{4}}D\\
	&\neq 0,
	\end{align*}
	where $ D $ is the non-vanishing determinant defined in \Cref{th:t1}, which has been calculated in \cref{appendix:app3} as \cref{eq:Dcalc}. Hence $ c_{2} = 0 $ and any $ c_{1} \neq 0 $ satisfies \cref{eq:nondegenerate}, so we do not need to consider the degenerate case. We now focus on computing:
	\begin{align}
	\det(AX' - aT') = \begin{vsmallmatrix}
	0 & -T' & 0 & X'\\
	A_{21}(X'-UT') & A_{22}(X'-UT') & A_{23}(X'-UT') + 2\tau_{*} T' & A_{24}(X'-UT')\\
	A_{31}(X'-UT') - a_{31}T' & A_{32}(X'-UT') - a_{32}T' & A_{33}(X'-UT') - a_{33}T' & A_{34}(X'-UT') - a_{34}T'\\
	A_{41}(X'-UT') + 2\tau_{*} W_{H}T' & A_{42}(X'-UT') & A_{43}(X'-UT') - 2\tau_{*} MT' & A_{44}(X'-UT')
	\end{vsmallmatrix}. \label{eq:bigDet}
	\end{align}
	\Cref{eq:bigDet} can be written in block form:
	\begin{align}
	\det(AX' - aT') = \begin{vmatrix}
	P_{11} & P_{12}\\
	P_{21} & P_{22}
	\end{vmatrix}, \label{eq:detBlockForm}
	\end{align}
	with
	\begin{align*}
	P_{11} &= \begin{pmatrix}
	0 & -T'\\
	A_{21}(X'-UT') & A_{22}(X'-UT') 
	\end{pmatrix}\\
	P_{12} &= \begin{pmatrix}
	0 & X'\\
	A_{23}(X'-UT') + 2\tau_{*}T' & A_{24}(X'-UT')
	\end{pmatrix}\\
	P_{21} &= \begin{pmatrix}
	A_{31}(X'-UT') - a_{31}T' & A_{32}(X'-UT') - a_{32}T'\\
	A_{41}(X'-UT') + 2\tau_{*}W_{H}T' & A_{42}(X'-UT')
	\end{pmatrix}\\
	P_{22} &= \begin{pmatrix}
	A_{33}(X'-UT') - a_{33}T' & A_{34}(X'-UT') - a_{34}T'\\
	A_{43}(X'-UT') - 2\tau_{*}MT' & A_{44}(X'-UT')
	\end{pmatrix}.
	\end{align*}
	\begin{Lemma}\label{lemma:l1}
		Under the assumptions of \Cref{th:t1}, the matrix
		\begin{align*}
		P_{11} = \begin{pmatrix}
		0 & -T'\\
		A_{21}(X'-UT') & A_{22}(X'-UT') 
		\end{pmatrix}
		\end{align*}
		is invertible.
	\end{Lemma}
	\begin{proof}
		We note that
		\begin{align}
		\det{P_{11}} = A_{21}T'(X'-UT') = \sigma_{*}T'(X'-UT'). \label{eq:P11}
		\end{align}
		From \Cref{th:t1}, $ \sigma_{*} \neq 0 $, so $ P_{11} $ is invertible if and only if $ T' \neq 0 $ and $ X' \neq UT' $. If $ T' = 0 $, then
		\begin{align*}
		\det(AX' - aT') &= X'^{4}\det{A}\\
		&= \frac{-W_{HJ}\tau_{*}}{2^{9}\pi^{4}}X'^{4}D.
		\end{align*}
		By assumption, $ W_{HJ} \neq 0 $, $ W_{H} \neq 0 $ (this corresponds to the underlying periodic wave having non-zero period) and hence $ \tau_{*} = \frac{1}{2}W_{H}W_{HJ} \neq 0 $. Moreover, $ D \neq 0 $ from \Cref{th:t1}, hence for $ X',T' $ to solve \cref{eq:qlcharacteristics} with $ T' = 0 $ we have:
		\begin{align*}
		\det(AX' - aT') = 0 \implies T' = X' = 0,
		\end{align*}
		which we exclude as a trivial solution. In fact, if $ T'(s_{*}) = 0 $ for some value of $ s = s_{*} $, we have that $ X'(s_{*}) = 0 $ as well, meaning that the characteristic curve $ (X(s),T(s)) $ terminates at $ s = s_{*} $. This is not possible, since the characteristic curves are defined for all $ X,T \in \R $, so in fact $ T' $ can never vanish. Next, if $ X'=UT' $, then:
		\begin{align*}
		\det(AX' - aT') &= \begin{vmatrix}
		0 & -T' & 0 & UT'\\
		0 & 0 & 2\tau_{*} T' & 0\\
		-a_{31}T' & - a_{32}T' & - a_{33}T' & - a_{34}T'\\
		2\tau_{*} W_{H}T' & 0 & - 2\tau_{*} MT' & 0
		\end{vmatrix}\\
		&= -2\tau_{*}T'\begin{vmatrix}
		0 & -T' & UT'\\
		-a_{31}T' & - a_{32}T' & -a_{34}T'\\
		2\tau_{*} W_{H}T' & 0 &  0
		\end{vmatrix}\\
		&= -4\tau_{*}^{2}W_{H}T'^{2}(a_{34}T'^{2} + a_{32}UT'^{2})\\
		&= -4\tau_{*}^{2}W_{H}T'^{4}\left(W_{H}^{2} + UM_{*}W_{HJ} - UW_{H}M_{*J} + U(W_{H}M_{*J} - W_{HJ}M_{*})\right)\\
		&= -4\tau_{*}^{2}W_{H}^{3}T'^{4}.
		\end{align*}
		Only $ T' $ is able to vanish, hence if $ X' = UT' $ then
		\begin{align*}
		\det(AX' - aT') = 0 \implies T' = X' = 0.
		\end{align*}
		There is no non-trivial solution $ (X',T') $ of \cref{eq:qlcharacteristics} for which $ P_{11} $ is singular, which proves the lemma.
	\end{proof}
	\Cref{lemma:l1} allows us to use the Schur determinant formula (see, for example, \cite{Zhang2006}) in \cref{eq:detBlockForm}, in particular:
	\begin{align}
	\det(AX'-aT') &= \det(P_{11})\det(P_{22} - P_{21}P_{11}^{-1}P_{12}). \label{eq:schurcomp}
	\end{align}
	Direct calculation of $ P_{22} - P_{21}P_{11}^{-1}P_{12} $ yields:
	\begin{align*}
	P_{22} - P_{21}P_{11}^{-1}P_{12} &= \begin{pmatrix}
	m_{11} & m_{12}\\
	m_{21} & m_{22}
	\end{pmatrix},
	\end{align*}
	where
	\begin{align}
	m_{11} &= A_{33}(X'-UT') - a_{33}T' - \frac{1}{\sigma_{*}(X'-UT')}(A_{23}(X'-UT') + 2\tau_{*}T')(A_{31}(X'-UT') - a_{31}T') \label{eq:m11}\\
	m_{12} &= A_{34}(X'-UT') - a_{34}T' + \frac{X'}{T'}(A_{32}(X'-UT')-a_{32}T') - \frac{A_{24}}{\sigma_{*}}(A_{31}(X'-UT') - a_{31}T')\\
	&- \frac{A_{22}X'}{\sigma_{*}T'}(A_{31}(X'-UT') - a_{31}T') \label{eq:m12}\\
	m_{21} &= A_{43}(X'-UT') -2\tau_{*}MT' - \frac{1}{\sigma_{*}(X'-UT')}(A_{23}(X'-UT') + 2\tau_{*}T')(A_{41}(X'-UT') + 2\tau_{*}W_{H}T') \label{eq:m21}\\
	m_{22} &= \frac{1}{T'}(X'-UT')(A_{42}X' + A_{44}T') -\frac{A_{24}}{\sigma_{*}}(A_{41}(X'-UT') + 2\tau_{*}W_{H}T') - \frac{A_{22}X'}{\sigma_{*}T'}(A_{41}(X'-UT') + 2\tau_{*}W_{H}T'). \label{eq:m22}
	\end{align}
	We make the substitution $ \lambda = X'-UT' $ and $ \mu = iT' $ (justified by the end result) and note from \cref{eq:schurcomp} that:
	\begin{align*}
	\det(AX'-aT') &= \sigma_{*}(-i\lambda\mu)\begin{vmatrix}
	m_{11} & m_{12}\\
	m_{21} & m_{22}
	\end{vmatrix}\\
	&= \frac{1}{\sigma_{*}W_{H}}\begin{vmatrix}
	\sigma_{*}\lambda W_{H}m_{11} & -i\sigma_{*}\mu W_{H}m_{12}\\
	\sigma_{*}\lambda m_{21} & -i\sigma_{*}\mu m_{22}
	\end{vmatrix},
	\end{align*}
	which can be simplified to:
	\begin{align}
	\det(AX'-aT') &= \frac{1}{\sigma_{*}W_{H}}\begin{vmatrix}
	a'_{11}\lambda^{2} + b'_{11}\lambda\mu + c'_{11}\mu^{2} & a'_{12}\lambda^{2} + b'_{12}\lambda\mu + c'_{12}\mu^{2}\\
	a'_{21}\lambda^{2} + b'_{21}\lambda\mu + c'_{21}\mu^{2} & a'_{22}\lambda^{2} + b'_{22}\lambda\mu + c'_{22}\mu^{2}
	\end{vmatrix}. \label{eq:matrixofquads}
	\end{align}
	We calculate the coefficients in \cref{appendix:app5}. To list them, we have:
	\begin{align*}
	a'_{11} &= 2\tau_{*}\left(\beta + \frac{U}{2}\right)^{2}(\Gamma M_{*H} + \rho_{*}M_{*J} + \sigma_{*}M_{*\gamma_{*}}) -2\tau_{*}\rho_{*}W_{H}\left(\beta + \frac{U}{2}\right) - \sigma_{*}\tau_{*}M_{*} - \nu_{*}\tau_{*}W_{H}\\
	b'_{11} &= 2i\tau_{*}W_{H}(2\tau_{*} + 2\Gamma \left(\beta + \frac{U}{2}\right) + J\sigma_{*})\\
	c'_{11} &= -2\tau_{*}W_{H}^{2}W_{JJ}\\
	a'_{12} &= \rho_{*}\tau_{*}W_{H} - 2\tau_{*}\left(\beta + \frac{U}{2}\right)(\Gamma M_{*H} + \rho_{*}M_{*J} + \sigma_{*}M_{*\gamma_{*}})\\
	b'_{12} &= 2i\tau_{*}W_{H}(\nu_{*} - \Gamma + \rho_{*}\left(\beta + \frac{U}{2}\right))\\
	c'_{12} &= 4\tau_{*}^{2}W_{H}\\
	a'_{21} &= \rho_{*}\tau_{*}W_{H} - 2\tau_{*}\left(\beta + \frac{U}{2}\right)(\Gamma M_{*H} + \rho_{*}M_{*J} + \sigma_{*}M_{*\gamma_{*}}) = a'_{12}\\
	b'_{21} &= 2i\tau_{*}W_{H}(\nu_{*} - \Gamma + \rho_{*}\left(\beta + \frac{U}{2}\right)) = b'_{12}\\
	c'_{21} &= 4\tau_{*}^{2}W_{H} = c'_{12}\\
	a'_{22} &= 2\tau_{*}(\Gamma M_{*H} + \rho_{*} M_{*J} + \sigma_{*} M_{*\gamma_{*}})\\
	b'_{22} &= -4i\tau_{*}\rho_{*}W_{H}\\
	c'_{22} &= 4\tau_{*}\nu_{*}W_{H}.
	\end{align*}
	We note that the matrix in \cref{eq:matrixofquads} is symmetric, which simplifies our calculations. Now we perform row and column operations to \cref{eq:matrixofquads}, which will leave the determinant unchanged. Adding $ (\beta + \frac{U}{2})  $ times the second row to the first row, and then $ (\beta + \frac{U}{2}) $ times the second column to the first column, we have the following result:
	\begin{align}
	\det(AX'-aT') &= \frac{1}{\sigma_{*}W_{H}}\begin{vmatrix}
	Q_{11}(\lambda,\mu) & Q_{12}(\lambda,\mu)\\
	Q_{21}(\lambda,\mu) & Q_{22}(\lambda,\mu)
	\end{vmatrix}. \label{eq:Qs}
	\end{align}
	The polynomials are calculated in \cref{appendix:app5}. We list them as:
	\begin{align*}
	Q_{11}(\lambda,\mu) &= \frac{\tau_{*}}{2^{15}\pi^{5}\sigma_{*}}(d_{2}\lambda^{2} + d_{1}\lambda\mu + d_{0}\mu^{2})\\
	Q_{12}(\lambda,\mu) &= \frac{\tau_{*}}{2^{15}\pi^{5}\sigma_{*}}(b_{2}\lambda^{2} + b_{1}\lambda\mu + b_{0}\mu^{2})\\
	Q_{21}(\lambda,\mu) &= Q_{12}(\lambda,\mu)\\
	Q_{22}(\lambda,\mu) &= \frac{\tau_{*}}{2^{15}\pi^{5}\sigma_{*}}(a_{2}\lambda^{2} + a_{1}\lambda\mu + a_{0}\mu^{2})
	\end{align*}
	We can swap the two columns and then the two rows without changing the determinant:
	\begin{align}
	\det(AX'-aT') &= \frac{1}{\sigma_{*}W_{H}}\begin{vmatrix}
	Q_{22}(\lambda,\mu) & Q_{21}(\lambda,\mu)\\
	Q_{12}(\lambda,\mu) & Q_{11}(\lambda,\mu)
	\end{vmatrix}\\
	&= \frac{\tau_{*}^{2}}{2^{30}\pi^{10}\sigma_{*}^{3}W_{H}}\begin{vmatrix}
	a_{2}\lambda^{2} + a_{1}\lambda\mu + a_{0}\mu^{2} & b_{2}\lambda^{2} + b_{1}\lambda\mu + b_{0}\mu^{2}\\
	b_{2}\lambda^{2} + b_{1}\lambda\mu + b_{0}\mu^{2} & d_{2}\lambda^{2} + d_{1}\lambda\mu + d_{0}\mu^{2}
	\end{vmatrix}\\
	&= \frac{W_{H}W_{HJ}^{2}}{2^{32}\pi^{10}\sigma_{*}^{3}}\det\left(\lambda^{2}\begin{pmatrix}
	a_{2} & b_{2}\\
	b_{2} & d_{2}
	\end{pmatrix} + \lambda\mu \begin{pmatrix}
	a_{1} & b_{1}\\
	b_{1} & d_{1}
	\end{pmatrix} + \mu^{2} \begin{pmatrix}
	a_{0} & b_{0}\\
	b_{0} & d_{0}
	\end{pmatrix}\right). \label{eq:finalDetWHJ}
	\end{align}
	This concludes the calculation of $ \det(AX'-aT') $ when $ W_{HJ} \neq 0 $. For completeness, we must carry out the same calculation starting with the assumption that $ W_{JJ} \neq 0 $ and $ W_{HH} \neq 0 $, which is the other case arising from the genericity condition $ \sigma_{*} \neq 0 $. This second calculation is unremarkable, so we give the final result here, but provide some working in \cref{appendix:app6}. In this case, we have, almost exactly as before:
	\begin{align}
	\det(AX'-aT') = -\frac{W_{H}W_{JJ}^{2}}{2^{32}\pi^{10}\sigma_{*}^{3}} \det\left(\lambda^{2}\begin{pmatrix}
	a_{2} & b_{2}\\
	b_{2} & d_{2}
	\end{pmatrix} + \lambda\mu \begin{pmatrix}
	a_{1} & b_{1}\\
	b_{1} & d_{1}
	\end{pmatrix} + \mu^{2} \begin{pmatrix}
	a_{0} & b_{0}\\
	b_{0} & d_{0}
	\end{pmatrix}\right), \label{eq:finalDetWJJ}
	\end{align}
	with the difference being the non-zero factor $ W_{JJ}^{2} $.
	This leads us to the following theorem.
	\begin{Theorem}\label{th:maintheorem}
		Suppose that the assumptions of \Cref{th:t1} hold, that is, $ \sigma \neq 0 $ and $ D \neq 0 $. Further, assume that $ T_{*} = 4\pi W_{H} \neq 0 $. Then, the equation for the characteristics of the Whitham modulation equations associated with the nonlinear Schr{\"o}dinger \cref{eq:1DNLS} is equivalent to the normal form for the continuous bands of spectrum emerging from the origin in the spectral plane, given in \Cref{th:t2}.
	\end{Theorem}
	\begin{proof}
		To solve for the characteristics of the Whitham system, we substitute into \cref{eq:qlcharacteristics} the appropriate expression for $ \det(AX'-aT') $ given in either \cref{eq:finalDetWHJ} or \cref{eq:finalDetWJJ}. In both cases,  we can divide $ \det(AX'-aT') $ by the constants we have assumed to be non-zero, so $ X'$ and $ T' $ solve the quartic
		\begin{align*}
		\det\left(\lambda^{2}\begin{pmatrix}
		a_{2} & b_{2}\\
		b_{2} & d_{2}
		\end{pmatrix} + \lambda\mu \begin{pmatrix}
		a_{1} & b_{1}\\
		b_{1} & d_{1}
		\end{pmatrix} + \mu^{2} \begin{pmatrix}
		a_{0} & b_{0}\\
		b_{0} & d_{0}
		\end{pmatrix}\right) = 0,
		\end{align*}
		with
		\begin{align}
		\lambda = X' - UT',\quad \mu = iT'. \label{eq:charstransformation}
		\end{align}
		This is exactly the normal form given in \Cref{th:t2}.
	\end{proof}
	\begin{remark}
		The substitution $ \mu = iT' $ implies that $ T' \in i\R $, however this is not required for $ (X',T') $ to be a solution to \cref{eq:qlcharacteristics}. Rather, we can safely make the restriction that $ T' \in i\R $ by multiplying \cref{eq:qlcharacteristics} by a phase function $ e^{4if(s)} $ so that $ T' $ becomes purely imaginary and $ X'\mapsto X'e^{if(s)} $ is possibly complex. \label{re:r1}
	\end{remark}
	\begin{corollary}
		Whitham modulation theory predicts the same criterion for modulational instability as the linear theory given in \Cref{cor:c1}, that is, the existence of a complex characteristic of the system of modulation equations (\cref{eq:qlsystem}) corresponds to a root $ \lambda $ of the normal form \cref{eq:lnormalform} with $ \mathrm{Re}(\lambda) \neq 0 $, indicating modulational instability of the underlying wave. \label{cor:c2}
	\end{corollary}
	\begin{proof}
		Noting that $ \det A \neq 0 $ and $ T'(s) \neq 0 $, we may rewrite \cref{eq:qlcharacteristics} as:
		\begin{align}
		\det\left(\frac{X'}{T'} - A^{-1}a\right) = 0. \label{eq:simplifiedchars}
		\end{align}
		Since $ T'(s) \neq 0 $ for any $ s $, the characteristic curves may instead be parametrized by $ T $, so that $ \frac{dX}{dT} = \frac{X'(s)}{T'(s)} $. Hence the eigenvalues $ \frac{X'(s)}{T'(s)} $ of $ A^{-1}a $ define the characteristic curves, which is clear once \cref{eq:qlsystem} is instead written as:
		\begin{align*}
		\begin{pmatrix}
		J_{T}\\\gamma_{*T}\\U_{T}\\\beta_{T}
		\end{pmatrix} + A^{-1}a\begin{pmatrix}
		J_{X}\\\gamma_{*X}\\U_{X}\\\beta_{X}
		\end{pmatrix} = 0,
		\end{align*}
		and similarly for \cref{eq:qlsystem2}. According to the Whitham theory \cite{Whi1999}, modulational instability occurs when one of the characteristics of a Whitham system is complex. From \cref{eq:simplifiedchars}, this is equivalent to $ \frac{X'}{T'} \in \C $. We have
		\begin{align*}
		\frac{X'}{T'} &= \frac{\lambda - iU\mu}{-i\mu}\\
		&= \frac{i\lambda}{\mu} + U.
		\end{align*}
		We note that $ \mu \in \R $ since it is the Floquet exponent of the periodic solutions to \cref{eq:spectralProblem}, and from \cref{re:r1} we know that this does not restrict the class of characteristic curves. Hence we have that
		\begin{align*}
		\frac{X'}{T'} \in \C\iff \mathrm{Re}(\lambda) \neq 0,
		\end{align*}
		which agrees with \Cref{cor:c1}.
	\end{proof}
	\begin{remark}
		The transformation from the characteristic variables of the Whitham system $ (X',T') $ to the spectral variable $ \lambda $ and Floquet exponent $ \mu $ may have greater significance in the scope of rigorously proving the agreement of Whitham modulation theory with linear stability at the origin. In fact, the transformation in all previous examples \cite{Serre2005,OZ2006,JZ2010,Jones14,JP2020} is:
		\begin{align*}
		\lambda = X'-cT',\quad \mu = -\frac{iT'}{T_{*}},
		\end{align*}
		where $ T_{*} $ is the period of the underlying wave. To explain the factor of $ \frac{1}{T_{*}} $, the cited papers all consider the Floquet multiplier to have the form $ e^{i\mu} $, whereas we have chosen to consider $ e^{i\mu T_{*}} $ in keeping with \cite{LBJM2019}.
	\end{remark}
	\section{Discussion and open problems}\label{sec:s4}
	In this paper, we show that the formal Whitham modulation theory correctly predicts the modulational instability of periodic, travelling wave solutions of the general nonlinear Schr{\"o}dinger \cref{eq:1DNLS} as prescribed by the rigorous spectral analysis at the origin in \cite{LBJM2019}. Applying the variational principle to the averaged Lagrangian allows us to derive four modulation equations, which we then homogenize using the genericity conditions described in \cite{LBJM2019}. This results in two cases depending on which slowly-varying parameters we eliminate, however the calculations are essentially the same. Finally, we compute a quartic equation for the characteristics of the homogenized modulation equations from the determinant of the quasi-linear system \cref{eq:qlsystem}. By invoking various determinant identities inspired by the proof of \cite[Proposition 1]{LBJM2019} and also a change of variables, we deduce that the characteristics of the Whitham system satisfy the same quartic equation as the normal form for the four continuous bands of spectrum at the origin \cite{LBJM2019}.\par
	Leisman et al. also provide a modulational instability criterion for transverse perturbations, where they consider the stability of periodic solutions of \cref{eq:1DNLS} to the two-dimensional equation:
	\begin{align*}
	i\psi_{t} &= \psi_{xx} \pm \psi_{zz} + \zeta f(\lvert\psi\rvert^{2})\psi.
	\end{align*}
	The extension is neat from the perspective of the linear theory, however we believe extending the Whitham theory would involve a two-phase approach, making the homogenization process considerably more difficult. We have decided that this lies outside the scope of this paper.\par
	One interesting future direction for our results would be to compute Riemann invariants and investigate the behaviour of dispersive shock waves for suitable one-dimensional potentials $ f(\lvert\psi\rvert^{2}) $. This would involve diagonalizing the matrix $ A^{-1}a $ from the quasi-linear system \cref{eq:qlsystem}.\par
	Another obvious extension would be to consider a coupled general nonlinear Schr{\"o}dinger system:
	\begin{align*}
	i\psi_{1t} &= \psi_{1xx}  + \zeta_{1} f(\zeta_{1}\lvert\psi_{1}\rvert^{2}, \zeta_{2}\lvert\psi_{2}\rvert^{2})\psi_{1}\\
	i\psi_{2t} &= \psi_{2xx} + \zeta_{2} f(\zeta_{1}\lvert\psi_{1}\rvert^{2}, \zeta_{2}\lvert\psi_{2}\rvert^{2})\psi_{2},
	\end{align*}
	which has the Lagrangian:
	\begin{align*}
	L = i(\overline{\psi_{1}}\psi_{1t} - \psi_{1}\overline{\psi}_{1t}) + i(\overline{\psi_{2}}\psi_{2t} - \psi_{2}\overline{\psi}_{2t}) + 2\lvert\psi_{1x}\rvert^{2} + 2\lvert\psi_{2x}\rvert^{2} - 2 F(\zeta_{1}\lvert\psi_{1}\rvert^{2},\zeta_{2}\lvert\psi_{2}\rvert^{2}).
	\end{align*}
	More generally, we could couple $ n $ general nonlinear Schr{\"o}dinger equations together, which yields
	\begin{align}
	i\partial_{t}\boldsymbol{\psi} &= \partial_{x}^{2}\boldsymbol{\psi} + \nabla_{\boldsymbol{\psi}} F(\lvert\psi_{1}\rvert^{2},\dots,\lvert\psi_{n}\rvert^{2})\\
	L &= i(\overline{\boldsymbol{\psi}}\cdot \partial_{t}\boldsymbol{\psi} - \boldsymbol{\psi}\cdot \partial_{t}\overline{\boldsymbol{\psi}}) + 2\lvert\partial_{x}\boldsymbol{\psi}\rvert^{2} - 2F(\lvert\psi_{1}\rvert^{2},\dots,\lvert\psi_{n}\rvert^{2}),
	\end{align}
	where $ F $ is a scalar potential. We believe that the Whitham theory for these systems would closely resemble the procedure for \cref{eq:1DNLS}, however difficulty may arise in the homogenization process. In particular, one must find suitable genericity conditions to eliminate enough slowly-varying parameters.\par
	In the context of the recent work of Chen et al. \cite{CPW2020}, \Cref{th:maintheorem} implies that a set of complex characteristics of the Whitham modulation equations supports rogue wave solutions localized on the background of periodic standing wave solutions of the focusing NLS equation. On the other hand, it appears that the existence of real characteristics is a necessary (but not sufficient) condition for observing algebraic solitons travelling on the background of periodic standing waves. More research is needed before this connection is rigorously established.\par
	Finally, we cite the open problem of proving that the modulational instability criterion derived from the hyperbolicity of a Whitham modulation system coincides with spectral instability of a periodic travelling wave solution of some general class of PDEs. The difficulty in developing a general proof lies in the homogenization of the Whitham equations. This process is guided by genericity conditions which are determined from the underlying PDE in the course of the rigorous linearized spectral analysis; this explains the disparity amongst the analyses of \cite{Serre2005,OZ2006,JZ2010,Jones14,JP2020} and this paper. On the other hand, it is promising that same the transformation between the characteristics of the Whitham system and the spectral variables (\cref{eq:charstransformation}) appears in the cited examples where the Whitham theory has been rigorously verified.
	\section{Acknowledgements}
	Both authors acknowledge the support of the Australian Research Council under grant DP200102130.
	\begin{appendices}\crefalias{section}{appsec}
		\section{Poisson brackets from the linear theory}\label{appendix:app1}
		We first define the Poisson brackets:
		\begin{align*}
		\gamma &= \{T_{*},\eta\}_{\kappa,\omega} \quad \rho = \{T_{*},\eta\}_{\omega,E}\quad \tau = \frac{T_{*}T_{*\kappa}}{2}\\
		\nu &= -\frac{T_{*}T_{*E}}{2} \quad \xi =  \{T_{*},\eta\}_{\kappa,\zeta}\quad \chi = \{T_{*},\eta\}_{\zeta,E}.
		\end{align*}
		With these definitions, we can list the matrix entries in \cref{eq:lnormalform}:
		\begin{align*}
		a_{2} &= -\frac{\sigma}{2}(\gamma M_{E} + \rho M_{\kappa} + \sigma M_{\omega})\\
		b_{2} &= -\frac{\sigma\rho T_{*}}{2} = -\sigma(\tau M_{E} + \nu M_{\kappa})\\
		d_{2} &= -\frac{\sigma}{2}(\nu T_{*} + \frac{\sigma}{2}M)\\
		a_{1} &= 2i\sigma\rho T_{*}\\
		b_{1} &= i\sigma T_{*}(\nu + \gamma)\\
		d_{1} &= i\sigma T_{*}(2\tau + \sigma\kappa)\\
		a_{0} &= 2\sigma\nu T_{*}\\
		b_{0} &= 2\sigma\tau T_{*}\\
		d_{0} &= 2\sigma T_{*}(\omega\gamma - \zeta\xi - E\sigma)
		\end{align*}
		\section{Relations between $ K $ and $ W $}\label{appendix:app2}
		In this appendix, we give the relations between the derivatives of $ K $ and the derivatives of $ W $, defined in \cref{eq:classicalAction,eq:W} respectively. This is carried out by taking a derivative of \cref{eq:KW} with respect to one of the parameters $ H,U,J,\gamma,\beta,\zeta $ and then applying the chain rule, simplifying the results using \cref{eq:omega,eq:kappa,eq:E}. Finally, we use \cref{eq:TNLS,eq:etaNLS,eq:MNLS} to eliminate $ K $. For the period $ T_{*} = K_{E} $ in the linear theory:
		\begin{align*}
		T_{*} &= 4\pi W_{H}\\
		T_{*E} &= 16\pi W_{HH}\\
		T_{*\kappa} &= 16\pi(W_{HJ} - (\beta + \frac{U}{2})W_{HH}) = -16\pi\eta_{*H}\\
		T_{*\omega} &= -4\pi M_{H}.
		\end{align*}
		Next, for $ \eta $ we have:
		\begin{align*}
		\eta &= 4\pi\eta_{*}\\
		\eta_{E} &= 16\pi\eta_{*H}\\
		\eta_{\kappa} &= 16\pi(\eta_{*J} - (\beta + \frac{U}{2})\eta_{*H})\\
		\eta_{\omega} &= -4\pi\eta_{*\gamma_{*}}.
		\end{align*}
		Finally for $ M $:
		\begin{align*}
		M &= 2\pi M_{*}\\
		M_{E} &= 8\pi M_{*H}\\
		M_{\kappa} &= 8\pi(M_{*J} - (\beta + \frac{U}{2})M_{*H})\\
		M_{\omega} &= -2\pi M_{*\gamma_{*}}.
		\end{align*}
		Since $ \zeta $ is the same parameter for both cases, we have that:
		\begin{align*}
		K_{\zeta} = \pi W_{\zeta}.
		\end{align*}
		\section{Poisson brackets for the modulation equations}\label{appendix:app3}
		Similarly to the Poisson brackets in \cref{appendix:app1}, we make the following definitions for Poisson brackets for the Whitham system:
		\begin{align*}
		\sigma_{*} &= \{W_{H},\eta_{*}\}_{H,J} = \{W_{H},W_{J}\}_{J,H}\\
		\rho_{*} &= \{W_{H},\eta_{*}\}_{\gamma_{*},H} = \{W_{H},W_{J}\}_{H,\gamma_{*}}\\
		\Gamma &= \{W_{H},\eta_{*}\}_{J,\gamma_{*}} = \{W_{H},W_{J}\}_{\gamma_{*},J}\\
		\nu_{*} &= -\frac{1}{2}W_{H}W_{HH}\\
		\tau_{*} &= \frac{1}{2}W_{H}W_{HJ}\\
		\xi_{*} &= \{W_{H},\eta_{*}\}_{J,\zeta} = \{W_{H},W_{J}\}_{\zeta,J}\\
		\chi_{*} &= \{W_{H},\eta_{*}\}_{\zeta,H} = \{W_{H},W_{J}\}_{H,\zeta}.
		\end{align*}
		The reason we have included the Poisson brackets in terms of $ W_{J} $ as well as $ \eta_{*} $ is because it is more straightforward to use the symmetry of mixed partial derivatives of $ W $ when using $ W_{J} $.\par
		We make extensive use of several properties of Poisson brackets. Let $ P,Q,S $ be functions of $ w,x,y,z $ with symmetric mixed partial derivatives. Then:
		\begin{itemize}
			\item $ \{P,Q\}_{w,x} = -\{P,Q\}_{x,w} = -\{Q,P\}_{w,x} $;
			\item $ \{P_{w},Q_{x}\}_{y,z} = \{P_{y},Q_{z}\}_{w,x} $;
			\item $ \{P,P\}_{w,x} = 0 $;
			\item $ \{PQ,S\}_{w,x} = P\{Q,S\}_{w,x} + Q\{P,S\}_{w,x} $;
			\item $ \{\alpha P,Q\}_{w,x} = \alpha\{P,Q\}_{w,x} $ for a constant $ \alpha $.
		\end{itemize}
		The above definitions allow us to express the equivalent Poisson brackets in \cref{appendix:app1} in terms of these newly-defined Poisson brackets for the Whitham parameters:
		\begin{align*}
		\sigma &= 256\pi^{2}\sigma_{*}\\
		\rho &= -64\pi^{2}\rho_{*}\\
		\gamma &= -64\pi^{2}(\Gamma + \left(\beta + \frac{U}{2}\right)\rho_{*})\\
		\nu &= 64\pi^{2}\nu_{*}\\
		\tau &= 64\pi^{2}(\tau_{*} + \left(\beta + \frac{U}{2}\right)\nu_{*})\\
		\chi &= 64\pi^{2}\chi_{*}\\
		\xi &= 64\pi^{2}(\xi_{*} + \left(\beta + \frac{U}{2}\right)\chi_{*})
		\end{align*}
		Ultimately, we use these expressions to express the matrix entries of \cref{eq:lnormalform} (given in \cref{appendix:app1}) in terms of the Whitham parameters:
		\begin{align*}
		a_{2} &= 2^{16}\pi^{5}\sigma_{*}(\Gamma M_{*H} + \rho_{*} M_{*J} + \sigma_{*} M_{*\gamma_{*}})\\
		b_{2} &= 2^{15}\pi^{5}\rho_{*}\sigma_{*}W_{H} = -2^{16}\sigma_{*}(\tau_{*} M_{*H} + \nu_{*} M_{*J})\\
		d_{2} &= -2^{15}\pi^{5}\sigma_{*}(\nu_{*} W_{H} + \sigma_{*} M_{*})\\
		a_{1} &= -2^{17}i\pi^{5}\sigma_{*}\rho_{*} W_{H}\\
		b_{1} &= 2^{16}i\pi^{5}\sigma_{*}W_{H}(\nu_{*} - \Gamma - \left(\beta + \frac{U}{2}\right)\rho_{*})\\
		d_{1} &= 2^{16}i\pi^{5}\sigma_{*}W_{H}(2\tau_{*} + 2\left(\beta + \frac{U}{2}\right)\nu_{*} + J\sigma_{*})\\
		a_{0} &= 2^{17}\pi^{5}\sigma_{*}\nu_{*}W_{H}\\
		b_{0} &= 2^{17}\pi^{5}\sigma_{*}W_{H}(\tau_{*} + \left(\beta + \frac{U}{2}\right)\nu_{*})
		\end{align*}
		For the last entry $ d_{0} $, we use the identity:
		\begin{align}
		W = 2W_{H}\left(H + \beta J + \frac{UJ}{2}\right) - 2M_{*}\left(\beta U + \frac{U^{2}}{4} - \gamma_{*}\right) - J\eta_{*} - 2\zeta W_{\zeta}. \label{eq:Widentity}
		\end{align}
		In particular, taking Poisson brackets of \cref{eq:Widentity} with $ W_{J} $ and derivatives with respect to $ H $ and $ J $, and using the fact that $ W_{J} = 0 $ in \cref{eq:Bdisprel} yields:
		\begin{align*}
		W_{H}\eta_{*J} &= -\sigma_{*}\left(2H + \beta J + \frac{UJ}{2}\right) - 2\Gamma\left(\beta U + \frac{U^{2}}{4} - \gamma_{*}\right) - 2\zeta\xi_{*}.
		\end{align*}
		Similarly, we have for a Poisson bracket of \cref{eq:Widentity} with $ W_{H} $ and derivatives $ H $ and $ J $:
		\begin{align*}
		W_{H}\eta_{*H} &= 2\rho_{*}\left(\beta U + \frac{U^{2}}{4} - \gamma_{*}\right) + J\sigma_{*} + 2\zeta\chi_{*}.
		\end{align*}
		Hence
		\begin{align*}
		d_{0} &= 2\sigma T_{*}(\omega\gamma - \zeta\xi - E\sigma)\\
		&= 2^{17}\pi^{5}\sigma_{*}W_{H}(-\left(\beta U + \frac{U^{2}}{4} - \gamma_{*}\right)\Gamma - \zeta\xi_{*} - \left(H + \beta J + \frac{UJ}{2}\right)\sigma_{*} - \left(\beta + \frac{U}{2}\right)(\rho_{*}\left(\beta + \frac{U^{2}}{4} - \gamma_{*}\right) + \zeta\chi_{*}))\\
		&= 2^{16}\pi^{5}\sigma_{*}W_{H}^{2}(\eta_{*J} - \eta_{*H}\left(\beta + \frac{U}{2}\right)).
		\end{align*}
		Now writing $ \eta_{*H},\eta_{*J} $ as the derivatives of $ W_{J} $ using its definition in \cref{eq:WJ}, we see that
		\begin{align*}
		d_{0} &= 2^{16}\pi^{5}\sigma_{*}W_{H}^{2}(2W_{HJ}\left(\beta + \frac{U}{2}\right) - W_{JJ} - \left(\beta + \frac{U}{2}\right)^{2}W_{HH})\\
		&=2^{16}\pi^{5}\sigma_{*}W_{H}(4\tau_{*}\left(\beta + \frac{U}{2}\right) + 2\nu_{*}\left(\beta + \frac{U}{2}\right)^{2} - W_{H}W_{JJ}).
		\end{align*}
		We also provide a calculation of the determinant $ D $ from \cref{th:t1} in terms of the Whitham parameters:
		\begin{align*}
		D &= -\frac{4}{\sigma^{3}}(a_{2}d_{2} - b_{2}^{2})\\
		&= -\frac{4}{2^{24}\pi^{6}\sigma_{*}^{3}}\left((2^{16}\pi^{5}\sigma_{*}(\Gamma M_{*H} + \rho_{*}M_{*J} + \sigma_{*}M_{*\gamma}))(-2^{15}\pi^{5}\sigma_{*}(\nu_{*}W_{H} + \sigma_{*}M_{*})) - (2^{16}\pi^{5}\sigma_{*}(\tau_{*}M_{*H} + \nu_{*}M_{*J}))^{2}\right)\\
		&= -\frac{2^{9}\pi^{4}}{\sigma_{*}}\left(-(\Gamma M_{*H} + \rho_{*}M_{*J} + \sigma_{*}M_{*\gamma_{*}})(\nu_{*}W_{H} + \sigma_{*}M_{*}) - 2(\tau_{*}M_{*H} + \nu_{*}M_{*J})^{2}\right)\\
		&= \frac{2^{9}\pi^{4}}{\sigma_{*}}\left(\sigma_{*}M_{*}\left(\Gamma M_{*H} + \rho_{*}M_{*J} + \sigma_{*}M_{*\gamma_{*}}\right) + \nu_{*}W_{H}\left(\Gamma M_{*H} + \rho_{*}M_{*J} + \sigma_{*}M_{*\gamma_{*}}\right) + \frac{W_{H}^{2}}{2}\left(W_{HJ}M_{*H} - W_{HH}M_{*J}\right)^{2}\right)\\
		&= 2^{9}\pi^{4}\bigg(M_{*}\left(\Gamma M_{*H} + \rho_{*} M_{*J} + \sigma_{*}M_{*\gamma_{*}}\right)\\
		&+ \frac{W_{H}^{2}}{2\sigma_{*}}\left(-W_{HH}\left(\Gamma M_{*H} + \rho_{*}M_{*J} + \sigma_{*}M_{*\gamma_{*}}\right) + W_{HJ}^{2}M_{*H}^{2} - 2W_{HJ}W_{HH}M_{*H}M_{*J} + W_{HH}^{2}M_{*J}^{2}\right)\bigg)\\
		&= 2^{9}\pi^{4}\bigg(M_{*}\left(\Gamma M_{*H} + \rho_{*}M_{*J} + \sigma_{*}M_{*\gamma}\right) - \frac{W_{H}^{2}}{2\sigma_{*}}\Big(\sigma_{*}W_{HH}M_{*\gamma} + W_{HH}M_{*H}(M_{*H}W_{JJ} - W_{HJ}M_{*J})\\
		&+ W_{HH}M_{*J}(W_{HH}M_{*J} - W_{HJ}M_{*H}) - W_{HJ}^{2}M_{*H}^{2} + 2W_{HJ}W_{HH}M_{*H}M_{*J} - W_{HH}^{2}M_{*J}^{2}\Big)\bigg)\\
		&= 2^{9}\pi^{4}\bigg(M_{*}\left(\Gamma M_{*H} + \rho_{*}M_{*J} + \sigma_{*}M_{*\gamma_{*}}\right) - \frac{W_{H}^{2}}{2\sigma_{*}}\left(\sigma_{*}W_{HH}M_{*\gamma_{*}} - M_{*H}^{2}(W_{HJ}^{2} - W_{HH}W_{JJ})\right)\bigg)\\
		&= 2^{9}\pi^{4}\bigg(M_{*}\left(\Gamma M_{*H} + \rho_{*} M_{*J} + \sigma_{*} M_{*\gamma_{*}}\right) - \frac{W_{H}^{2}}{2\sigma_{*}}\left(\sigma_{*} W_{HH}M_{*\gamma_{*}} - \sigma_{*} M_{*H}^{2}\right)\bigg)
		\end{align*}
		from which we have:
		\begin{align}
		D = 2^{9}\pi^{4}\left(M_{*}\left(\Gamma M_{*H} + \rho_{*} M_{*J} + \sigma_{*} M_{*\gamma_{*}}\right) - \frac{W_{H}^{2}}{2}\{W_{H},W_{\gamma_{*}}\}_{H,\gamma_{*}}\right). \label{eq:Dcalc}
		\end{align}
		\section{Calculating matrix elements}\label{appendix:app4}
		In this appendix, we provide some calculations for the matrix elements of \cref{eq:matrixA,eq:matrixa}. We already calculated $ A_{21},A_{22},A_{23},A_{24} $ in \cref{eq:vecComp}. For \cref{eq:mod3Final}, we compute:
		\begin{align*}
		W_{HJ}\partial_{T}W_{U} &= \{W_{U},W_{J}\}_{J,H}J_{T} + \{W_{U},W_{J}\}_{\gamma_{*},H}\gamma_{*T} + \{W_{U},W_{J}\}_{U,H}U_{T} + \{W_{U},W_{J}\}_{\beta,H}\beta_{T}\\
		&= \bigg\{\frac{JW_{H}}{2} - W_{\gamma_{*}}\left(\beta + \frac{U}{2}\right),W_{J}\bigg\}_{J,H}J_{T} + \bigg\{\frac{JW_{H}}{2} - W_{\gamma_{*}}\left(\beta + \frac{U}{2}\right),W_{J}\bigg\}_{\gamma_{*},H}\gamma_{*T}\\
		&+ \bigg\{\frac{JW_{H}}{2} - W_{\gamma_{*}}\left(\beta + \frac{U}{2}\right),W_{J}\bigg\}_{U,H}U_{T} + \bigg\{\frac{JW_{H}}{2} - W_{\gamma_{*}}\left(\beta + \frac{U}{2}\right),W_{J}\bigg\}_{\beta,H}\beta_{T}\\
		&= \left(\tau_{*} + \frac{J}{2}\sigma_{*} + \Gamma\left(\beta + \frac{U}{2}\right)\right )J_{T} + \left(-\frac{J}{2}\rho_{*} - \left(\beta + \frac{U}{2}\right)\{W_{H},W_{\gamma_{*}}\}_{J,\gamma_{*}}\right)\gamma_{*T}\\
		&+ \left(\frac{J}{2}\bigg\{\frac{JW_{H}}{2} - W_{\gamma_{*}}\left(\beta + \frac{U}{2}\right),W_{H}\bigg\}_{H,J} - \frac{1}{2}M_{*}W_{HJ} - \left(\beta + \frac{U}{2}\right)\bigg\{\frac{JW_{H}}{2} - W_{\gamma_{*}}\left(\beta + \frac{U}{2}\right),W_{H}\bigg\}_{\gamma_{*},J}\right)U_{T}\\
		&+ \left(\frac{J}{2}\big\{JW_{H} - UW_{\gamma_{*}},W_{H}\big\}_{H,J} - M_{*}W_{HJ} - \left(\beta + \frac{U}{2}\right)\big\{JW_{H} - UW_{\gamma_{*}},W_{H}\big\}_{\gamma_{*},J}\right)\beta_{T}\\
		&= \left(\tau_{*} + \frac{J}{2}\sigma_{*} + \Gamma\left(\beta + \frac{U}{2}\right)\right )J_{T} + \left(-\frac{J}{2}\rho_{*} - \left(\beta + \frac{U}{2}\right)\{W_{H},W_{\gamma_{*}}\}_{J,\gamma_{*}}\right)\gamma_{*T}\\
		&+ \left(\frac{J}{2}\nu_{*} + \frac{J}{2}\rho_{*}\left(\beta + \frac{U}{2}\right) - \frac{1}{2}M_{*}W_{HJ} + \frac{1}{2}W_{H}M_{*H}\left(\beta + \frac{U}{2}\right) + \left(\beta + \frac{U}{2}\right)^{2}\{W_{H},W_{\gamma_{*}}\}_{J,\gamma_{*}}\right)U_{T}\\
		&+ \left(J\nu_{*} + \frac{J}{2}U\rho_{*} - M_{*}W_{HJ} + W_{H}M_{H}\left(\beta + \frac{U}{2}\right) + U\left(\beta + \frac{U}{2}\right)\{W_{H},W_{\gamma_{*}}\}_{J,\gamma_{*}}\right)\beta_{T}\\
		&= A_{31}J_{T} + A_{32}\gamma_{*T} + A_{33}U_{T} + A_{34}\beta_{T}.
		\end{align*}
		For $ a_{31},a_{32},a_{33},a_{34} $:
		\begin{align*}
		W_{HJ}W_{U}U_{X} - W_{HJ}\partial_{X}W &= W_{HJ}W_{U}U_{X} - \{W,W_{J}\}_{J,H}J_{X} - \{W,W_{J}\}_{\gamma_{*},H}\gamma_{*X} - \{W,W_{J}\}_{U,H}U_{X} - \{W,W_{J}\}_{\beta,H}\beta_{X}\\
		&= W_{H}W_{JJ}J_{X} - (M_{*}W_{HJ} - W_{H}M_{*J})\gamma_{*X} + W_{H}W_{UJ}U_{X} - (W_{\beta}W_{HJ} - W_{H}W_{\beta J})\beta_{X}.
		\end{align*}
		Taking the $ J $-derivatives of $ W_{U} $ and $ W_{\beta} $ defined in \cref{eq:WU,eq:Wbeta} respectively, we have:
		\begin{align*}
		W_{HJ}W_{U}U_{X} - W_{HJ}\partial_{X}W &= W_{H}W_{JJ}J_{X} - (M_{*}W_{HJ} - W_{H}M_{*J})\gamma_{*X} + W_{H}\left(\frac{1}{2}W_{H} + \frac{J}{2}W_{HJ} - M_{*J}\left(\beta + \frac{U}{2}\right)\right)U_{X}\\
		&- \left(W_{HJ}(JW_{H} - UM_{*}) - W_{H}(W_{H} + JW_{HJ} - UM_{*J})\right )\beta_{X}\\
		&= W_{H}W_{JJ}J_{X} - (M_{*}W_{HJ} - W_{H}M_{*J})\gamma_{*X} + \left(\frac{1}{2}W_{H}^{2} + J\tau_{*} - W_{H}M_{*J}\left(\beta + \frac{U}{2}\right)\right)U_{X}\\
		&- \left(-W_{H}^{2} + UW_{H}M_{*J} - UM_{*}W_{HJ}\right)\beta_{X}\\
		&= a_{31}J_{X} + a_{32}\gamma_{*X} + a_{33}U_{X} + a_{34}\beta_{X}.
		\end{align*}
		Finally, in \cref{eq:mod4Final} we compute:
		\begin{align*}
		W_{HJ}W_{H}\partial_{T}M_{*} &= W_{H}\{W_{H},W_{J}\}_{J,\gamma_{*}}J_{T} + W_{H}\{W_{H},W_{\gamma_{*}}\}_{J,\gamma_{*}}\gamma_{*T} + W_{H}\{W_{H},W_{U}\}_{J,\gamma_{*}}U_{T} + W_{H}\{W_{H},W_{\beta}\}_{J,\gamma_{*}}\beta_{T}\\
		&= -W_{H}\Gamma J_{T} + W_{H}\{W_{H},W_{\gamma_{*}}\}_{J,\gamma_{*}}\gamma_{*T} + W_{H}\bigg\{W_{H},\frac{JW_{H}}{2} - W_{\gamma_{*}}\left(\beta + \frac{U}{2}\right)\bigg\}_{J,\gamma_{*}}U_{T}\\
		&+ W_{H}\big\{W_{H},JW_{H} - UW_{\gamma_{*}}\big\}_{J,\gamma_{*}}\beta_{T}\\
		&= -W_{H}\Gamma J_{T} + W_{H}\{W_{H},W_{\gamma_{*}}\}_{J,\gamma_{*}}\gamma_{*T} + W_{H}\left(-\frac{1}{2}W_{H}M_{*H} - \left(\beta + \frac{U}{2}\right)\{W_{H},W_{\gamma_{*}}\}_{J,\gamma_{*}}\right)U_{T}\\
		&+ W_{H}\left(-W_{H}M_{*H} - U\{W_{H},W_{\gamma_{*}}\}_{J,\gamma_{*}}\right)\beta_{T}
		\end{align*}
		Hence for $ A_{41},A_{42},A_{43},A_{44} $, we have:
		\begin{align*}
		W_{HJ}W_{H}\partial_{T}M_{*} - W_{HJ}M_{*}\partial_{T}W_{H} &= \left(-\Gamma W_{H} - \sigma_{*}M_{*}\right)J_{T} + \left(W_{H}\{W_{H},W_{\gamma_{*}}\}_{J,\gamma_{*}} + \rho_{*}M_{*}\right)\\
		&+ \left(-\frac{1}{2}W_{H}^{2}M_{*H} - W_{H}\left(\beta + \frac{U}{2}\right)\{W_{H},W_{\gamma_{*}}\}_{J,\gamma_{*}} - M_{*}\left(\nu_{*} + \rho_{*}\left(\beta + \frac{U}{2}\right)\right)\right)\\
		&+ \left(-W_{H}^{2}M_{*H} - UW_{H}\{W_{H},W_{\gamma_{*}}\}_{J,\gamma_{*}} - M_{*}(2\nu_{*} + U\rho_{*})\right)\\
		&= A_{41}J_{T} + A_{42}\gamma_{*T} + A_{43}U_{T} + A_{44}\beta_{T}.
		\end{align*}
		The $ X $-derivatives of \cref{eq:mod4Final} are:
		\begin{align*}
		&U(W_{HJ}W_{H}\partial_{X}M_{*} - W_{HJ}M_{*}\partial_{X}W_{H}) + W_{HJ}W_{H}M_{*}U_{X} - W_{HJ}W_{H}^{2}J_{X}\\
		&= (UA_{41} - 2\tau_{*}W_{H})J_{X} + UA_{42}\gamma_{*X} + (UA_{43} + 2\tau M_{*})U_{X} + UA_{44}\beta_{X}
		\end{align*}
		\section{Matrix elements for the Schur determinant calculation}\label{appendix:app5}
		In this appendix, we calculate the coefficients in \cref{eq:matrixofquads}, using the expressions for $ m_{11},m_{12},m_{21},m_{22} $ computed in \cref{eq:m11,eq:m12,eq:m21,eq:m22}. We derive a number of useful identities that we use freely in the subsequent calculations:
		\begin{align*}
		\sigma_{*}\{W_{H},W_{\gamma_{*}}\}_{J,\gamma_{*}} - \rho_{*}\Gamma &= \sigma_{*}(W_{HJ}M_{*\gamma_{*}} - M_{*H}M_{*J}) - (M_{*J}W_{HH} - M_{*H}W_{HJ})(M_{*H}W_{JJ} - M_{*J}W_{HJ})\\
		&= W_{HJ}\sigma_{*}M_{*\gamma_{*}} - M_{*H}M_{*J}(W_{HJ}^{2} - W_{HH}W_{JJ})\\
		&- \left(M_{*J}M_{*H}W_{HH}W_{JJ} - M_{*H}^{2}W_{HJ}W_{JJ} - M_{*J}^{2}W_{HH}W_{HJ} + M_{*H}M_{*J}W_{HJ}^{2}\right)\\
		&= W_{HJ}\left(\Gamma M_{*H} + \rho_{*}M_{*J} + \sigma_{*}M_{*\gamma_{*}}\right)\\
		\frac{1}{2}\sigma_{*}W_{H}M_{*H} - \nu_{*}\Gamma &= \frac{1}{2}W_{H}M_{*H}(W_{HJ}^{2} - W_{HH}W_{JJ}) + \frac{1}{2}W_{H}W_{HH}(M_{*H}W_{JJ} - M_{*J}W_{HJ})\\
		&= \frac{1}{2}W_{H}W_{HJ}(M_{*H}W_{HJ} - M_{*J}W_{HH})\\
		&= -\tau_{*}\rho_{*}\\
		\frac{W_{H}^{2}}{2}\sigma_{*} - \nu_{*}W_{H}W_{JJ} + 2\tau_{*}^{2} &= \frac{W_{H}^{2}}{2}(W_{HJ}^{2} - W_{HH}W_{JJ} + W_{HH}W_{JJ} + W_{HJ}^{2})\\
		&= W_{H}^{2}W_{HJ}^{2}\\
		&= 4\tau_{*}^{2}\\
		2\tau_{*}\Gamma - W_{H}W_{JJ}\rho_{*} - W_{H}M_{*J}\sigma_{*} &= W_{H}(W_{HJ}(M_{*H}W_{JJ} - M_{*J}W_{HJ}) - W_{JJ}(M_{*J}W_{HH} - M_{*H}W_{HJ})\\
		&- M_{*J}(W_{HJ}^{2} - W_{HH}W_{JJ}))\\
		&= 2W_{H}W_{HJ}(M_{*H}W_{JJ} - M_{*J}W_{HJ})\\
		&= 4\tau_{*}\Gamma\\
		\sigma_{*}M_{*J} + \rho_{*}W_{JJ} &= M_{*J}(W_{HJ}^{2} - W_{HH}W_{JJ}) + W_{JJ}(M_{*J}W_{HH} - M_{*H}W_{HJ})\\
		&= W_{HJ}(M_{*J}W_{HJ} - M_{*H}W_{JJ})\\
		&= -W_{HJ}\Gamma\\
		\sigma_{*}W_{H}M_{*H} - 2\nu_{*}\Gamma &= W_{H}\left(M_{*H}(W_{HJ}^{2} - W_{HH}W_{JJ}) + W_{HH}(M_{*H}W_{JJ} - M_{*J}W_{HJ})\right)\\
		&= W_{H}W_{HJ}(M_{*H}W_{HJ} - M_{*J}W_{HH})\\
		&= -2\tau_{*}\rho_{*}.
		\end{align*}
		For the coefficients $ a'_{11},b'_{11},c'_{11} $, we have:
		\begin{align*}
		a_{11}' &= W_{H}(\sigma_{*}A_{33} - A_{23}A_{31})\\
		&= W_{H}\bigg[\frac{J}{2}\nu_{*}\sigma_{*} + \frac{J}{2}\rho_{*}\sigma_{*}\left(\beta + \frac{U}{2}\right) - \frac{1}{2}\sigma_{*}M_{*}W_{HJ} + \frac{1}{2}\sigma_{*}W_{H}M_{*H}\left(\beta + \frac{U}{2}\right) + \sigma_{*}\left(\beta + \frac{U}{2}\right)^{2}\{W_{H},W_{\gamma_{*}}\}_{J,\gamma_{*}}\\
		&- \left(\nu_{*} + \rho_{*}\left(\beta +  \frac{U}{2}\right)\right)\left(\tau_{*} + \frac{J}{2}\sigma_{*} + \Gamma\left(\beta + \frac{U}{2}\right)\right )\bigg]\\
		&= W_{H}\bigg[\left(\beta + \frac{U}{2}\right)^{2}\left(\sigma_{*}\{W_{H},W_{\gamma_{*}}\}_{J,\gamma_{*}} -\rho_{*}\Gamma\right) +\left(\beta + \frac{U}{2}\right)\left(\frac{1}{2}\sigma_{*}W_{H}M_{*H} - \nu_{*}\Gamma -\tau_{*}\rho_{*}\right) - \frac{1}{2}\sigma_{*}M_{*}W_{HJ} - \nu_{*}\tau_{*}\bigg]\\
		&= W_{H}W_{HJ}\left(\beta + \frac{U}{2}\right)^{2}(\Gamma M_{*H} + \rho_{*}M_{*J} + \sigma_{*}M_{*\gamma_{*}}) - 2\tau_{*}\rho_{*}W_{H}\left(\beta + \frac{U}{2}\right) - \sigma_{*}\tau_{*}M_{*} - \nu_{*}\tau_{*}W_{H}\\
		b_{11}' &= iW_{H}(\sigma_{*}a_{33} - A_{23}a_{31} + 2\tau_{*}A_{31})\\
		&= iW_{H}\bigg(\frac{W_{H}^{2}}{2}\sigma_{*} + J\tau_{*}\sigma_{*} - W_{H}M_{*J}\sigma_{*}\left(\beta + \frac{U}{2}\right)  - W_{H}W_{JJ}\left(\nu_{*} + \rho_{*}\left(\beta + \frac{U}{2}\right)\right) + 2\tau_{*}\left(\tau_{*} + \frac{J}{2}\sigma_{*} + \Gamma\left(\beta + \frac{U}{2}\right)\right)\bigg)\\
		&= iW_{H}\bigg(\frac{W_{H}^{2}}{2}\sigma_{*} - \nu_{*}W_{H}W_{JJ} + 2\tau_{*}^{2} + 2J\tau_{*}\sigma_{*} + (2\tau_{*}\Gamma - W_{H}W_{JJ}\rho_{*} - W_{H}M_{*J}\sigma_{*})\left(\beta + \frac{U}{2}\right)\bigg)\\
		&= iW_{H}\bigg(4\tau_{*}^{2} + 2J\tau_{*}\sigma_{*} + 4\tau_{*}\Gamma\left(\beta + \frac{U}{2}\right)\bigg)\\
		c'_{11} &= -2\tau_{*}W_{H}a_{31}\\
		&= -2\tau_{*}W_{H}^{2}W_{JJ}
		\end{align*}
		Next for $ a'_{12},b'_{12},c'_{12} $:
		\begin{align*}
		a'_{12} &= W_{H}(\sigma_{*}A_{32} - A_{22}A_{31})\\
		&= W_{H}\left[-\frac{J}{2}\rho_{*}\sigma_{*} - \sigma_{*}\left(\beta + \frac{U}{2}\right)\{W_{H},W_{\gamma_{*}}\}_{J,\gamma_{*}} + \rho_{*}\left(\tau_{*} + \frac{J}{2}\sigma_{*} + \Gamma\left(\beta + \frac{U}{2}\right)\right)\right]\\
		&= W_{H}\left[\left(\beta + \frac{U}{2}\right)\left(\rho_{*}\Gamma - \sigma_{*}\{W_{H},W_{\gamma_{*}}\}_{J,\gamma_{*}}\right) + \tau_{*}\rho_{*}\right]\\
		&= \tau_{*}\rho_{*}W_{H} - 2\tau_{*}\left(\beta + \frac{U}{2}\right)(\Gamma M_{*H} + \rho_{*}M_{*J} + \sigma_{*}M_{*\gamma_{*}})\\
		b'_{12} &= -iW_{H}\left[\sigma_{*}A_{34} + \sigma_{*}UA_{32} - \sigma_{*}a_{32} - A_{24}A_{31} - UA_{22}A_{31} + A_{22}a_{31}\right]\\
		&= -iW_{H}\bigg[J\nu_{*}\sigma_{*} + \sigma_{*}W_{H}M_{*H}\left(\beta + \frac{U}{2}\right) - \sigma_{*}M_{*}W_{HJ} + \frac{J}{2}U\rho_{*}\sigma_{*} + U\sigma_{*}\left(\beta + \frac{U}{2}\right)\{W_{H},W_{\gamma_{*}}\}_{J,\gamma_{*}}\\
		&- \frac{J}{2}U\rho_{*}\sigma_{*} - U\sigma_{*}\left(\beta + \frac{U}{2}\right)\{W_{H},W_{\gamma_{*}}\}_{J,\gamma_{*}} - \sigma_{*}W_{H}M_{*J} + \sigma_{*}W_{HJ}M_{*} - (2\nu_{*} + U\rho_{*})\left(\tau_{*} + \frac{J\sigma_{*}}{2} + \Gamma\left(\beta + \frac{U}{2}\right)\right)\\
		&+ U\rho_{*}\left(\tau_{*} + \frac{J\sigma_{*}}{2} + \Gamma\left(\beta + \frac{U}{2}\right)\right) - \rho_{*}W_{H}W_{JJ}\bigg]\\
		&= -iW_{H}\left[-2\nu_{*}\tau_{*} - W_{H}(\sigma_{*}M_{*J} + \rho_{*}W_{JJ}) + \left(\beta + \frac{U}{2}\right)(\sigma_{*}W_{H}M_{*H} - 2\nu_{*}\Gamma)\right]\\
		&= 2i\tau_{*}W_{H}\left(\nu_{*} - \Gamma + \rho_{*}\left(\beta + \frac{U}{2}\right)\right)\\
		c'_{12} &= W_{H}\left[\sigma_{*}a_{34} + \sigma_{*}Ua_{32} - A_{24}a_{31} - UA_{22}a_{31}\right]\\
		&= W_{H}\left[\sigma_{*}(W_{H}^{2} + UM_{*}W_{HJ} - UW_{H}M_{*J}) + U\sigma_{*}(W_{H}M_{*J} - W_{HJ}M_{*}) - W_{H}W_{JJ}(2\nu_{*} + U\rho_{*}) + U\rho_{*}W_{H}W_{JJ}\right]\\
		&= W_{H}\left[\sigma_{*}W_{H}^{2} - 2\nu_{*}W_{H}W_{JJ}\right]\\
		&= W_{H}^{3}W_{HJ}^{2}\\
		&= 4\tau_{*}^{2}W_{H}.
		\end{align*}
		Now for $ a'_{21},b'_{21},c'_{21} $:
		\begin{align*}
		a'_{21} &= \sigma_{*}A_{43} - A_{23}A_{41}\\
		&= \sigma_{*}\left(-\frac{W_{H}^{2}M_{*H}}{2} -W_{H}\left(\beta + \frac{U}{2}\right)\{W_{H},W_{\gamma_{*}}\}_{J,\gamma_{*}} - M_{*}\left(\nu_{*} + \rho_{*}\left(\beta + \frac{U}{2}\right)\right)\right)\\
		&- \left(\nu_{*} + \rho_{*}\left(\beta +  \frac{U}{2}\right)\right)\left(-\Gamma W_{H} - \sigma_{*}M_{*}\right)\\
		&= W_{H}\left(\beta + \frac{U}{2}\right)(\rho_{*}\Gamma - \sigma_{*}\{W_{H},W_{\gamma_{*}}\}_{J,\gamma_{*}}) + \nu_{*}\Gamma W_{H} - \frac{W_{H}^{2}}{2}\sigma_{*}M_{*H}\\
		&= -2\tau\left(\beta + \frac{U}{2}\right)(\Gamma M_{*H} + \rho_{*}M_{*J} + \sigma_{*}M_{*\gamma_{*}}) + \tau_{*}\rho_{*}W_{H}\\
		b'_{21} &= i(2\tau_{*}\sigma_{*}M_{*} + 2\tau_{*}W_{H}A_{23} + 2\tau_{*}A_{41})\\
		&= 2i\tau_{*}\left(\sigma_{*}M_{*} + W_{H}\left(\nu_{*} + \rho_{*}\left(\beta +  \frac{U}{2}\right)\right) -\Gamma W_{H} - \sigma_{*}M_{*}\right)\\
		&= 2i\tau_{*}W_{H}\left(\nu_{*} - \Gamma + \rho_{*}\left(\beta + \frac{U}{2}\right)\right)\\
		c'_{21} &= 4\tau_{*}^{2}W_{H}
		\end{align*}
		Finally for $ a'_{22},b'_{22},c'_{22} $:
		\begin{align*}
		a'_{22} &= \sigma_{*}A_{42} - A_{22}A_{41}\\
		&= \sigma_{*}\left(W_{H}\{W_{H},W_{\gamma_{*}}\}_{J,\gamma_{*}} + \rho_{*}M_{*}\right) + \rho_{*}\left(-\Gamma W_{H} - \sigma_{*}M_{*}\right)\\
		&= W_{H}\left(\sigma_{*}\{W_{H},W_{\gamma_{*}}\}_{J,\gamma_{*}} - \rho_{*}\Gamma\right)\\
		&= 2\tau_{*}\left(\Gamma M_{*H} + \rho_{*}M_{*J} + \sigma_{*}M_{*\gamma_{*}}\right)\\
		b'_{22} &= -i\left(U\sigma_{*}A_{42} + \sigma_{*}A_{44} - A_{24}A_{41} - 2\tau_{*}W_{H}A_{22} - UA_{22}A_{41}\right)\\
		&= -i\bigg(U\sigma_{*}(W_{H}\{W_{H},W_{\gamma_{*}}\}_{J,\gamma_{*}} + \rho_{*}M_{*}) + \sigma_{*}\left(-W_{H}^{2}M_{*H} - UW_{H}\{W_{H},W_{\gamma_{*}}\}_{J,\gamma_{*}} -M_{*}(2\nu_{*} + U\rho_{*})\right)\\
		&- (2\nu_{*} + U\rho_{*})( -\Gamma W_{H} - \sigma_{*}M_{*}) + 2\tau_{*}\rho_{*}W_{H} + U\rho_{*}(-\Gamma W_{H} - \sigma_{*}M_{*})\bigg)\\
		&= -i\left(-\sigma_{*}W_{H}^{2}M_{*H} + 2\nu_{*}\Gamma W_{H} + 2\tau_{*}\rho_{*}W_{H}\right)\\
		&= -4i\tau_{*}\rho_{*}W_{H}\\
		c'_{22} &= 2\tau_{*}W_{H}A_{24} + 2\tau_{*}UW_{H}A_{22}\\
		&= 2\tau_{*}W_{H}\left(2\nu_{*} + U\rho_{*} - U\rho_{*}\right)\\
		&= 4\tau_{*}\nu_{*}W_{H}.
		\end{align*}
		For the calculation of the quadratics in \cref{eq:Qs}, we have:
		\begin{align*}
		Q_{11}(\lambda,\mu) &= \left[a'_{11} + \left(\beta + \frac{U}{2}\right)a'_{21} + \left(\beta + \frac{U}{2}\right)\left(a'_{12} + \left(\beta + \frac{U}{2}\right)a'_{22}\right)\right]\lambda^{2}\\
		&+ \left[b'_{11} + \left(\beta + \frac{U}{2}\right)b'_{21} + \left(\beta + \frac{U}{2}\right)\left(b'_{12} + \left(\beta + \frac{U}{2}\right)b'_{22}\right)\right]\lambda\mu\\
		&+ \left[c'_{11} + \left(\beta + \frac{U}{2}\right)c'_{21} + \left(\beta + \frac{U}{2}\right)\left(c'_{12} + \left(\beta + \frac{U}{2}\right)c'_{22}\right)\right]\mu^{2}\\
		&= -\tau_{*}\left[\nu_{*}W_{H} + \sigma_{*}M_{*}\right]\lambda^{2} + 2i\tau_{*}W_{H}\left[2\tau_{*} + 2\left(\beta + \frac{U}{2}\right)\nu_{*} + J\sigma_{*}\right]\lambda\mu\\
		&+ 2\tau_{*}W_{H}\left[4\tau_{*}\left(\beta + \frac{U}{2}\right) + 2\nu_{*}\left(\beta + \frac{U}{2}\right)^{2}-W_{H}W_{JJ}\right]\mu^{2}\\
		&= \frac{\tau_{*}}{2^{15}\pi^{5}\sigma_{*}}\left(d_{2}\lambda^{2} + d_{1}\lambda\mu + d_{0}\mu^{2}\right)
		\end{align*}
		Next:
		\begin{align*}
		Q_{12}(\lambda,\mu) &= \left[a'_{12} + \left(\beta + \frac{U}{2}\right)a'_{22}\right]\lambda^{2} + \left[b'_{12} + \left(\beta + \frac{U}{2}\right)b'_{22}\right]\lambda\mu + \left[c'_{12} + \left(\beta + \frac{U}{2}\right)c'_{22}\right]\mu^{2}\\
		&= \rho_{*}\tau_{*}W_{H}\lambda^{2} + 2i\tau_{*}W_{H}\left[\nu_{*} - \Gamma - \rho_{*}\left(\beta + \frac{U}{2}\right)\right]\lambda\mu + 4\tau_{*}W_{H}\left[\tau_{*} + \nu_{*}\left(\beta + \frac{U}{2}\right)\right]\mu^{2}\\
		&= \frac{\tau_{*}}{2^{15}\pi^{5}\sigma_{*}}\left(b_{2}\lambda^{2} + b_{1}\lambda\mu + b_{0}\mu^{2}\right).
		\end{align*}
		Using the symmetry of the matrix in \cref{eq:matrixofquads}, we have:
		\begin{align*}
		Q_{21}(\lambda,\mu) &= \left[a'_{21} + \left(\beta + \frac{U}{2}\right)a'_{22}\right]\lambda^{2} + \left[b'_{21} + \left(\beta + \frac{U}{2}\right)b'_{22}\right]\lambda\mu + \left[c'_{21} + \left(\beta + \frac{U}{2}\right)c'_{22}\right]\mu^{2}\\
		&= \left[a'_{12} + \left(\beta + \frac{U}{2}\right)a'_{22}\right]\lambda^{2} + \left[b'_{12} + \left(\beta + \frac{U}{2}\right)b'_{22}\right]\lambda\mu + \left[c'_{12} + \left(\beta + \frac{U}{2}\right)c'_{22}\right]\mu^{2}\\
		&= Q_{12}(\lambda,\mu).
		\end{align*}
		Finally:
		\begin{align*}
		Q_{22}(\lambda,\mu) &= a'_{22}\lambda^{2} + b'_{22}\lambda\mu + c'_{22}\mu^{2}\\
		&= \frac{\tau_{*}}{2^{15}\pi^{5}\sigma_{*}}\left(a_{2}\lambda^{2} + a_{1}\lambda\mu + a_{0}\mu^{2}\right).
		\end{align*}
		\section{Overview of the case $ W_{JJ} \neq 0 $}\label{appendix:app6}
		In this appendix, we provide an overview for finding the characteristics of the Whitham modulation \cref{eq:mod1Final,eq:mod2Final,eq:mod3Final,eq:mod4Final} when $ W_{JJ} $ and $ W_{HH} $ are assumed to be non-vanishing. This provides justification for \cref{eq:finalDetWJJ}. In a similar manner to \cref{eq:gimplicit}, we start by applying the implicit function theorem to $ W_{J} = 0 $, which yields a continuously differentiable function $ h $ such that:
		\begin{align*}
		J = h(H,U,\gamma_{*},\beta).
		\end{align*}
		Taking derivatives of $ W_{J} = 0 $ provides the relations:
		\begin{align*}
		\partial_{H}:\quad h_{H}W_{JJ} &= - W_{HJ}\\
		\partial_{U}:\quad h_{U}W_{JJ} &= - W_{UJ}\\
		\partial_{\gamma}:\quad h_{\gamma_{*}}W_{JJ} &= - W_{\gamma_{*} J}\\
		\partial_{\beta}:\quad h_{\beta}W_{JJ} &= - W_{\beta J}.
		\end{align*}
		We now apply the chain rule to the $ T $ and $ X $ derivatives in the modulation equations in order to write them in terms of derivatives of the parameters. With $ z = T,X $, the required derivatives are:
		\begin{align*}
		W_{JJ}\partial_{z}W_{H} &= \{W_{H},W_{J}\}_{H,J}H_{z} + \{W_{\gamma},W_{J}\}_{H,J}\gamma_{*z} + \{W_{U},W_{J}\}_{H,J}U_{z} + \{W_{\beta},W_{J}\}_{H,J}\beta_{z}\\
		&= -\sigma_{*}H_{z} + \Gamma \gamma_{*z} - \left(\tau_{*} + \frac{J}{2}\sigma_{*} + \Gamma\left(\beta + \frac{U}{2}\right)\right)U_{z} -  (2\tau_{*} + J\sigma_{*} + U\Gamma)\beta_{z}\\
		W_{JJ}\partial_{z}W_{U} &= \{W_{U},W_{J}\}_{H,J}H_{z} + \{W_{U},W_{J}\}_{\gamma_{*},J}\gamma_{*z} + \{W_{U},W_{J}\}_{U,J}U_{z} + \{W_{U},W_{J}\}_{\beta,J}\beta_{z}\\
		&= -\left(\tau_{*} + \frac{J}{2}\sigma_{*} + \Gamma\left(\beta + \frac{U}{2}\right)\right)H_{z}  + \left(-\frac{W_{H}M_{*J}}{2} + J\Gamma - \left(\beta + \frac{U}{2}\right)\{W_{\gamma_{*}},W_{J}\}_{\gamma_{*},J}\right)\gamma_{*z}\\
		&+ \left( -\frac{1}{4}W_{H}^{2} - J\tau_{*} - \frac{J^{2}}{4}\sigma_{*} - \frac{1}{2}M_{*}W_{JJ} + W_{H}M_{*J}\left(\beta + \frac{U}{2}\right) - J\Gamma\left(\beta + \frac{U}{2}\right) + \left(\beta + \frac{U}{2}\right)^{2}\{W_{\gamma_{*}},W_{J}\}_{\gamma_{*},J}\right)U_{z}\\
		&+ \left(-\frac{1}{2}W_{H}^{2} - 2J\tau_{*} - \frac{J^{2}}{2}\sigma_{*} - M_{*}W_{JJ} + W_{H}M_{*J}\left(\beta + U\right) - J\Gamma(\beta + U) + U\left(\beta + \frac{U}{2}\right)\{W_{\gamma_{*}},W_{J}\}_{\gamma_{*},J}\right)\beta_{z}\\
		W_{JJ}\partial_{z}W &= W_{H}W_{JJ}H_{z} + M_{*}W_{JJ}\gamma_{*z} + W_{U}W_{JJ}U_{z} + W_{JJ}(JW_{H} - UM_{*})\beta_{z}\\
		W_{JJ}\partial_{z}W_{\gamma_{*}} &= \{W_{\gamma_{*}},W_{J}\}_{H,J}H_{z} + \{W_{\gamma_{*}},W_{J}\}_{\gamma_{*},J}\gamma_{*z} + \{W_{\gamma_{*}},W_{J}\}_{U,J}U_{z} + \{W_{\gamma_{*}},W_{J}\}_{\beta,J}\beta_{z}\\
		&= \Gamma H_{z} + \{W_{\gamma_{*}},W_{J}\}_{\gamma_{*},J}\gamma_{*z} + \left(-\frac{W_{H}M_{*J}}{2} + J\Gamma - \left(\beta + \frac{U}{2}\right)\{W_{\gamma_{*}},W_{J}\}_{\gamma_{*},J}\right)U_{z}\\
		&+ \left(-W_{H}M_{*J} + J\Gamma - U\{W_{\gamma_{*}},W_{J}\}_{\gamma_{*},J}\right)\beta_{z}.
		\end{align*}
		As in \cref{eq:qlsystem}, we write the modulation equations in the quasi-linear form:
		\begin{align}
		A\begin{pmatrix}
		H_{T}\\\gamma_{*T}\\U_{T}\\\beta_{T}
		\end{pmatrix} - a\begin{pmatrix}
		H_{X}\\\gamma_{*X}\\U_{X}\\\beta_{X}
		\end{pmatrix} = 0 \label{eq:qlsystem2}
		\end{align}
		with
		\begin{align*}
		A &= \begin{pmatrix}
		0 & 0 & 0 & 1\\
		A_{21} & A_{22} & A_{23} & A_{24}\\
		A_{31} & A_{32} & A_{33} & A_{34}\\
		A_{41} & A_{42} & A_{43} & A_{44}
		\end{pmatrix}\\
		a &= \begin{pmatrix}
		0 & 1 & 0 & 0\\
		UA_{21} & UA_{22} & UA_{23} - W_{H}W_{JJ} & UA_{24}\\
		UA_{31} + a_{31} & UA_{32} + a_{32} & UA_{33} & UA_{34} + a_{34}\\
		UA_{41} + 2\tau W_{H} & UA_{42} + W_{H}^{2}M_{*J} & UA_{43} + a_{43} & UA_{44} + a_{44}
		\end{pmatrix}
		\end{align*}
		and the coefficients
		\begin{align*}
		A_{21} &= -\sigma_{*}\\
		A_{22} &= \Gamma\\
		A_{23} &= -\left(\tau_{*} + \frac{J}{2}\sigma_{*} + \Gamma\left(\beta + \frac{U}{2}\right)\right)\\
		A_{24} &= -(2\tau_{*} + J\sigma_{*} + U\Gamma)\\
		A_{31} &= -\left(\tau_{*} + \frac{J}{2}\sigma_{*} + \Gamma\left(\beta + \frac{U}{2}\right)\right)\\
		A_{32} &= -\frac{W_{H}M_{*J}}{2} + \frac{J}{2}\Gamma -\left(\beta + \frac{U}{2}\right)\{W_{\gamma_{*}},W_{J}\}_{\gamma_{*},J}\\
		A_{33} &= \left( -\frac{1}{4}W_{H}^{2} - J\tau_{*} - \frac{J^{2}}{4}\sigma_{*} - \frac{1}{2}M_{*}W_{JJ} + W_{H}M_{*J}\left(\beta + \frac{U}{2}\right) - J\Gamma\left(\beta + \frac{U}{2}\right) + \left(\beta + \frac{U}{2}\right)^{2}\{W_{\gamma_{*}},W_{J}\}_{\gamma_{*},J}\right)\\
		A_{34} &= \left(-\frac{1}{2}W_{H}^{2} - 2J\tau_{*} - \frac{J^{2}}{2}\sigma_{*} - M_{*}W_{JJ} + W_{H}M_{*J}\left(\beta + U\right) - J\Gamma(\beta + U) + U\left(\beta + \frac{U}{2}\right)\{W_{\gamma_{*}},W_{J}\}_{\gamma_{*},J}\right)\\
		A_{41} &= \Gamma W_{H} + \sigma_{*}M_{*}\\
		A_{42} &= W_{H}\{W_{\gamma_{*}},W_{J}\}_{\gamma_{*},J} - \Gamma M_{*}\\
		A_{43} &= W_{H}\left(-\frac{W_{H}M_{*J}}{2} + \frac{J}{2}\Gamma -\left(\beta + \frac{U}{2}\right)\{W_{\gamma_{*}},W_{J}\}_{\gamma_{*},J}\right) + M_{*}\left(\tau_{*} + \frac{J}{2}\sigma_{*} + \Gamma\left(\beta + \frac{U}{2}\right)\right)\\
		A_{44} &= W_{H}(-W_{H}M_{*J} + J\Gamma - U\{W_{\gamma_{*}},W_{J}\}_{\gamma_{*},J}) + M_{*}(2\tau_{*} + J\sigma_{*} + U\Gamma)\\
		a_{31} &= -W_{H}W_{JJ}\\
		a_{32} &=-M_{*}W_{JJ}\\
		a_{34} &= -W_{JJ}(JW_{H} - UM_{*})\\
		a_{43} &= M_{*}W_{H}W_{JJ} + W_{H}^{2}\left(\frac{W_{H}}{2} + \frac{JW_{HJ}}{2} - M_{*J}\left(\beta + \frac{U}{2}\right)\right)\\
		a_{44} &= W_{H}^{2}(W_{H} + JW_{HJ} - UM_{*J}).
		\end{align*}
		The equation for the characteristics is
		\begin{align*}
		AX' - aT' &= 0.
		\end{align*}
		The matrices $ A,a $ satisfy \cref{eq:nondegenerate}, since we can compute:
		\begin{align}
		\det(AX') &= \frac{W_{H}W_{JJ}^{2}}{2^{10}\pi^{4}}X'^{4}D \neq 0. \label{eq:AX}
		\end{align}
		Moreover, we can follow the same procedure of applying the Schur determinant formula. The upper-left block matrix is:
		\begin{align*}
		P_{11} = \begin{pmatrix}
		0 & -T'\\
		-\sigma_{*}(X'-UT') & \Gamma(X'-UT')
		\end{pmatrix},
		\end{align*}
		which is only singular when $ X' = T' = 0 $ (from \cref{eq:AX}), or when $ X' = UT' $. However, $ X' = UT' $ leads to the determinant calculation:
		\begin{align*}
		\det(UAT' - aT') = W_{H}^{5}W_{JJ}^{2}T'^{4},
		\end{align*}
		which only vanishes when $ T' = 0 $, leading to the trivial solution $ X' = T' = 0 $. From here, the procedure of using the Schur determinant formula and manipulating the determinant calculation is the same as in the Section 3.1. The end result is given in \cref{eq:finalDetWJJ}.
	\end{appendices}
	
	\bibliographystyle{amsalpha}  
	

\begin{thebibliography}{1}
		
		\bibitem[AS1981]{AS1981}
		{\sc M.J.~Ablowitz and H.~Segur},
		{\em Solitons and the inverse scattering transform},
		SIAM (1981), Philadelphia.
		
		\bibitem[Agrawal2013]{Agrawal2013}
		{\sc G.P.~Agrawal},
		{\em Nonlinear fiber optics},
		Elsevier/Academic Press (2013), Amsterdam, Fifth edition.
		
		\bibitem[AIK1990]{AIK1990}
		{\sc G.L.~Alfimov, A.R.~Its and N.E.~Kulagin},
		{\em Modulation instability of solutions of the nonlinear Schr{\"o}dinger equation},
		Teoret. Mat. Fiz. \textbf{84} (1990), no. 2, 787--793.
		
		\bibitem[BGNR2014]{BGNR2014}
		{\sc S.~Benzoni-Gavage, P.~Noble and L.M.~Rodrigues},
		{\em Slow modulations of periodic waves in Hamiltonian PDEs, with application to capillary fluids},
		J. Nonlinear Sci. \textbf{24} (2014), no. 4, 711--768.
		
		\bibitem[BGMR2016]{BGMR2016}
		{\sc S.~Benzoni-Gavage, C.~Mietka and L.M.~Rodrigues},
		{\em Co-periodic stability of periodic waves in some Hamiltonian PDEs},
		Nonlinearity \textbf{29} (2016), no. 11, 3241--3308.
		
		\bibitem[Bridges2015]{Bridges2015}
		{\sc T.J.~Bridges},
		{\em Breakdown of the Whitham modulation theory and the emergence of dispersion},
		Stud. Appl. Math. \textbf{135} (2015), no. 3, 277--294.
		
		\bibitem[BKZ2021]{BKZ2021}
		{\sc T.J.~Bridges, A.~Kostianko and S.~Zelik},
		{\em Validity of the hyperbolic Whitham modulation equations in Sobolev spaces},
		J. Differential Equations \textbf{274} (2021), 971--995.
		
		\bibitem[BKS2020]{BKS2020}
		{\sc T.J.~Bridges, A.~Kostianko and G.~Schneider},
		{\em A proof of validity for multiphase Whitham modulation theory},
		Proc. R. Soc. Lond. A \textbf{476} (2020), no. 22443, 20200203.
		
		\bibitem[BR2019]{BR2019}
		{\sc T.J.~Bridges and D.J.~Ratliff},
		{\em Krein signatures and Whitham modulation theory: the sign of characteristics and the 'sign characteristic'},
		Stud. Appl. Math. \textbf{142} (2019), no. 3, 314--335.
		
		\bibitem[BJ2010]{BJ2010}
		{\sc J.C.~Bronski and M.A.~Johnson},
		{\em The modulational instability for a generalized Korteweg-de Vries equation},
		Arch. Mech. Rat. Anal. \textbf{197} (2010), no. 2, 357--400.
		
		\bibitem[BDN2011]{BDN2011}
		{\sc N.~Bottman, B.~Deconinck and M.~Nivala},
		{\em Elliptic solutions of the defocusing NLS equation are stable},
		J. Phys. A. \textbf{44} (2011), no. 28, 1--24.
		
		\bibitem[Chen2016]{Chen2016}
		{\sc F.F.~Chen},
		{\em Introduction to plasma physics and controlled fusion},
		Cham: Springer International Publishing (2016). Reprint of the 1984 edition.
		
		\bibitem[CPW2020]{CPW2020}
		{\sc J.~Chen, D.E.~Pelinovsky and R.E.~White},
		{\em Periodic standing waves in the focusing nonlinear Schr{\"o}dinger equation: Rogue waves and modulation instability},
		Phys. D \textbf{405} (2020), 132378.
		
		\bibitem[DS2017]{DS2017}
		{\sc B.~Deconinck and B.L.~Segal},
		{\em The stability spectrum for elliptic solutions to the focusing NLS equation},
		Phys. D \textbf{346} (2017), 1--19.
		
		\bibitem[DS2009]{DS2009}
		{\sc W.-P.~D{\"u}ll and G.~Schneider},
		{\em Validity of Whitham's equations for the modulation of periodic traveling waves in the NLS equation},
		J. Nonlinear Sci. \textbf{19} (2009), no. 5, 453--466.
		
		\bibitem[EH2016]{EH2016}
		{\sc G.A.~El and M.A.~Hoefer},
		{\em Dispersive shock waves and modulation theory},
		Phys. D \textbf{333} (2016), 11--65.
		
		\bibitem[GH2007A]{GH2007A}
		{\sc T.~Gallay and M.~H\v{a}r\v{a}gu{\c{s}}},
		{\em Stability of small periodic waves for the nonlinear Schr{\"o}dinger equation},
		J. Differential Equations \textbf{243} (2007), no. 2, 544--581.
		
		\bibitem[GH2007B]{GH2007B}
		{\sc T.~Gallay and M.~H\v{a}r\v{a}gu{\c{s}}},
		{\em Orbital stability of periodic waves for the nonlinear Schr{\"o}dinger equation},
		J. Dynam. Differential Equations \textbf{19} (2007), no. 4, 825--865
		
		\bibitem[GSS1987]{GSS1987}
		{\sc M.~Grillakis, J.~Shatah and W.~Strauss},
		{\em Stability theory of solitary waves in the presence of symmetry, I},
		J. Func. Anal. \textbf{74} (1987), no. 1, 160--197.
		
		\bibitem[GSS1990]{GSS1990}
		{\sc M.~Grillakis, J.~Shatah and W.~Strauss},
		{\em Stability theory of solitary waves in the presence of symmetry, II},
		J. Func. Anal. \textbf{94} (1990), no. 2, 308--348.
		
		\bibitem[Gross1961]{Gross1961}
		{\sc E.P.~Gross},
		{\em Structure of a quantized vortex in boson systems},
		Nuovo Cimento \textbf{20} (1961), no. 3, 454--477.
		
		\bibitem[GLCT2017]{GLCT2017}
		{\sc S.~Gustafson, S.~Le Coz and T.-P.~Tsai},
		{\em Stability of periodic waves of 1D cubic nonlinear Schr{\"o}dinger equations},
		Appl. Math. Res. Express. AMRX, \textbf{2017} (2017), no. 2, 431--487.
		
		\bibitem[HM2003]{HM2003}
		{\sc A.~Hasegawa and M.~Matsumoto},
		{\em Optical solitons in fibers},
		Springer, (2003), Berlin. Third edition.
		
		\bibitem[HO1972]{HO1972}
		{\sc H.~Hasimoto and H.~Ono},
		{\em Nonlinear modulation of gravity waves},
		J. Phys. Soc. Japan \textbf{33} (1972), no. 3, 805--811.
		
		\bibitem[JMMP2014]{Jones14}
		{\sc C.K.R.T.~Jones, R.~Marangell, P.D.~Miller and R.G.~Plaza},
		{\em Spectral and modulational stability of periodic wavetrains for the nonlinear Klein-Gordon equation},
		J. Differential Equations {\bf 257} (2014), no. 12, 4632--4703.
		
		\bibitem[JP2020]{JP2020}
		{\sc M.A.~Johnson and W.R.~Perkins},
		{\em Modulational instability of viscous fluid conduit periodic waves},
		SIAM J. Math. Anal. \textbf{52} (2020), no. 1, 277--305.
		
		\bibitem[JZ2010]{JZ2010}
		{\sc M.A.~Johnson and K.~Zumbrun},
		{\em Rigorous justification of the Whitham modulation equations for the generalized Korteweg-de Vries equation},
		Stud. Appl. Math. \textbf{125} (2010), no. 1, 69--89.
		
		\bibitem[Kamchatnov2000]{Kamchatnov2000}
		{\sc A.M.~Kamchatnov},
		{\em Nonlinear periodic waves and their modulations: an introductory course},
		World Scientific Publishing (2000), Singapore.
		
		\bibitem[LBJM2019]{LBJM2019}
		{\sc K.P.~Leisman, J.C.~Bronski, M.A.~Johnson and R.~Marangell},
		{\em Stability of Traveling wave solutions of Nonlinear Dispersive equations of NLS type},
		Preprint (2019), arXiv:1910.05392.
		
		\bibitem[LTE2019]{LTE2019}
		{\sc C.S.~Liu, V.K.~Tripathi and B.~Eliasson},
		{\em High-power laser-plasma interaction},
		Cambridge University Press, (2019), Cambridge.
		
		\bibitem[MOMT1976]{MOMT1976}
		{\sc K.~Mio, T.~Ogino, K.~Minami and S.~Takeda},
		{\em Modified nonlinear Schr{\"o}dinger equation for Alfv{\'en} waves propagating along the magnetic field in cold plasmas},
		J. Phys. Soc. Japan \textbf{41} (1976), no. 1, 265--271.
		
		\bibitem[OZ2003]{OZ2003}
		{\sc M.~Oh and K.~Zumbrun},
		{\em Stability of periodic solutions of conservation laws with viscosity: analysis of the evans function},
		Arch. Ration. Mech. Anal. \textbf{166} (2003), no. 2, 99--166.
		
		\bibitem[OZ2006]{OZ2006}
		{\sc M.~Oh and K.~Zumbrun},
		{\em Low-frequency stability analysis of periodic travelling-wave solutions of viscous conservation laws in several dimensions},
		Z. Anal. Anwend. \textbf{25} (2006), no. 1, 1--21.
		
		\bibitem[Pitaevskii1961]{Pitaevskii1961}
		{\sc L.P.~Pitaevskii},
		{\em Vortex lines in an imperfect Bose gas},
		Sov. Phys.--JETP \textbf{13} (1961), no. 2, 451--454.
		
		\bibitem[Rowlands1974]{Rowlands1974}
		{\sc G.~Rowlands},
		{\em On the stability of solutions of the non-linear Schr{\"o}dinger equation},
		IMA J. Appl. Math. \textbf{13} (1974), no. 3, 367--377.
		
		\bibitem[Serre2005]{Serre2005}
		{\sc D.~Serre},
		{\em Spectral stability of periodic solutions of viscous conservation laws: large wavelength analysis},
		Communications in Partial Differential Equations \textbf{30} (2005), no. 1--2, 259--282.
		
		\bibitem[SS1999]{SS1999}
		{\sc C.~Sulem and P.-L.~Sulem},
		{\em The nonlinear Schr{\"o}dinger equation: self-focusing and wave collapse},
		Springer-Verlag, (1999), New York.
		
		\bibitem[Whi1965A]{Whi1965A}
		{\sc G.B.~Whitham},
		{\em Non-linear dispersive waves},
		Proc. R. Soc. Lond. A \textbf{283} (1965), no. 1393, 238--261.
		
		\bibitem[Whi1965B]{Whi1965B}
		{\sc G.B.~Whitham},
		{\em A general approach to linear and non-linear dispersive waves using a Lagrangian},
		Journal of Fluid Mechanics \textbf{22} (1965), no. 2, 273--283.
		
		\bibitem[Whi1967]{Whi1967}
		{\sc G.B.~Whitham},
		{\em Variational methods and applications to water waves},
		Proc. R. Soc. Lond. A \textbf{299} (1967), no. 1456, 6--25.
		
		\bibitem[Whi1970]{Whi1970}
		{\sc G.B.~Whitham},
		{\em Two-timing, variational principles and waves},
		Journal of Fluid Mechanics \textbf{44} (1970), no. 2, 373--395.
		
		\bibitem[Whi1999]{Whi1999}
		{\sc G.B.~Whitham},
		{\em Linear and nonlinear waves},
		John Wiley \& Sons, Inc. (1999), New York. Reprint of the 1974 original.
		
		\bibitem[Zakharov1968]{Zakharov1968}
		{\sc V.E.~Zakharov},
		{\em Stability of periodic waves of finite amplitude on the surface of a deep fluid},
		J. Appl. Mech. Tech. Phys. \textbf{9} (1968), no. 2, 190--194.
		
		\bibitem[Zhang2006]{Zhang2006}
		{\sc F.~Zhang},
		{\em The Schur complement and its applications},
		Springer Science \& Business Media (2006), vol. 4.
	\end{thebibliography}

\end{document}